\documentclass[11pt]{amsart}
\usepackage{amsmath,amsthm, amscd, amssymb, amsfonts, mathrsfs,pb-diagram,pb-xy}
\usepackage[all]{xy}
\usepackage[dvips, dvipsnames, usenames]{color}
\usepackage{enumitem}
\usepackage[section]{placeins}

\newcommand{\ydG}{{}^{\ku G}_{\ku G}\mathcal{YD}}

\usepackage{hhline}

\usepackage[active]{srcltx}

\newcommand{\mor}{\sim_{\text{Mor}}}

\newcommand{\vect}{\operatorname{vec}}

\newcommand\sn{\mathbb S_n}

\newcommand{\fd}{finite-dimensional}

\newcommand{\gap}{\textsf{GAP}}
\newcommand{\anupq}{\textsf{HAP}}

\newcommand\ad{\operatorname{ad}}

\newcommand\toba{\mathfrak B }

\newcommand{\gr}{\operatorname{gr}}
\newcommand{\trid}{\triangleright}

\newcommand{\Fc}{{\mathcal F}}

\newcommand{\ku}{\Bbbk}

\newcommand{\N}{{\mathbb N}}

\newcommand{\G}{{\mathbb G}}

\newcommand{\M}{{\mathcal M}}
\newcommand{\Q}{{\mathsf Q}}
\newcommand{\F}{{\mathbb F}}
\newcommand{\C}{{\mathcal C}}
\newcommand{\D}{{\mathcal D}}

\newcommand{\q}{{\mathbf q}}

\newcommand{\T}{{\mathcal T}}

\newcommand{\Ac}{{\mathcal A}}

\newcommand{\Oc}{{\mathcal O}}

\newcommand{\ydhh}{{}^{H}_{H}\mathcal{YD}}

\newcommand{\ydg}{{}^{\ku G}_{\ku G}\mathcal{YD}}
\newcommand{\ydk}{{}^{K}_{K}\mathcal{YD}}

\newcommand{\Ss}{\mathcal S}

\newcommand{\End}{\operatorname{End}}
\newcommand{\Aut}{\operatorname{Aut}}
\newcommand{\Out}{\operatorname{Out}}

\newcommand\Rep{\operatorname{Rep}}
\newcommand\Corep{\operatorname{Corep}}

\newcommand\opext{\operatorname{Opext}}

\numberwithin{equation}{section}
\theoremstyle{plain}

\newtheorem{theorem}{Theorem}[section]
\newtheorem{lema}[theorem]{Lemma}

\newtheorem{prop}[theorem]{Proposition}

\newtheorem{question}[theorem]{Question}
\theoremstyle{definition}
\newtheorem{definition}[theorem]{Definition}

\theoremstyle{remark}
\newtheorem{obs}[theorem]{Remark}

\newcommand\id{\operatorname{id}}

\newcommand\st{\mathbb S_3}
\newcommand\sk{\mathbb S_4}
\newcommand\sco{\mathbb S_5}

\newcommand\ac{\mathbb A_4}

\newcommand\s{\mathbb S}

\def\pf{\begin{proof}}
\def\epf{\end{proof}}

\theoremstyle{remark}

\newcounter{commentcounter}

\newcommand{\gtha}[6]{H^{#1}_{#5\, #6}(#2, #3, #4)}
\newcommand{\gth}[5]{H^{#1}_{#4 \, #5}(#2, #3)}

\begin{document}

\renewcommand{\baselinestretch}{1.2}

\thispagestyle{empty}
\title[Hopf algebras with the dual Chevalley property]
{Examples of finite-dimensional Hopf algebras with the dual Chevalley property}

\author[N. Andruskiewitsch, C. Galindo and M. M\"uller]
{Nicol\'as Andruskiewitsch, C\'esar Galindo and Monique M\"uller}

\thanks{N. A. and M. M.  were partially supported by
ANPCyT-Foncyt, CONICET and  Secyt (UNC).  C. G. was partially supported by the FAPA funds from
vicerrectoria de investigaciones de la Universidad de los Andes. Part of this work was done during visits of N. A. to the Universidad de los Andes}

\address{\noindent Facultad de Matem\'atica, Astronom\'{\i}a y F\'{\i}sica,
Universidad Nacional de C\'ordoba. CIEM -- CONICET. 
Medina Allende s/n (5000) Ciudad Universitaria, C\'ordoba,
Argentina}
\email{(andrus|mmuller)@famaf.unc.edu.ar}
\address{\noindent  Department of mathematics,
Universidad de los Andes.
Carrera 1 No. 18A -12, Bogot\'a, Colombia.
}
\email{cn.galindo1116@uniandes.edu.co}


\begin{abstract}
We present new Hopf algebras with the dual Chevalley property by determining all semisimple Hopf algebras  Morita-equivalent to a group algebra over a finite group, for a list of groups supporting a non-trivial \fd{} Nichols algebra.
\end{abstract}

\maketitle

\setcounter{tocdepth}{1}


\section{Introduction}\label{sec:sln}

\subsection{}\label{subsec:intro1}

A Hopf algebra has the dual Chevalley property if the tensor
product of two simple comodules is semisimple, or equivalently if
its coradical is a (cosemisimple) Hopf subalgebra. These Hopf 
algebras are interesting by various reasons,
among them  the Lifting Method for their classification
 \cite{AS-cambr}. Particular classes are the
pointed (the coradical is a group algebra) and copointed (the
coradical is the algebra of functions on a finite or reductive
group) Hopf algebras. However few examples out of these classes
has been discussed in the literature \cite{de1tipo6chevalley, Mom}. The purpose of this paper is
to present explicit examples of Hopf algebras with the
dual Chevalley property.

\subsection{}\label{subsec:intro2} In this paper
the underlying field  $\ku$ is algebraically closed of characteristic 0. 
Then the coradical of a \fd{} Hopf
algebra with the dual Chevalley property is a semisimple Hopf
algebra. Let us consider the problem of constructing (and
eventually classifying) \fd{} Hopf algebras $A$ whose coradical
$A_0$ is isomorphic to a fixed semisimple Hopf algebra $H$. For
this, we need to address three problems:

\begin{enumerate}\renewcommand{\theenumi}{\alph{enumi}}
\renewcommand{\labelenumi}{(\theenumi)}
\item\label{item:lefting-i} To find (classify) the Yetter-Drinfeld
modules $V \in \ydhh$ such that the dimension of the Nichols algebra
$\toba(V)$ is finite. Then $\Ac(V) = \toba(V)\# H$ is a \fd{} Hopf
algebra with $\Ac(V)_0 \simeq H$.

 \item\label{item:lefting-ii} To find (classify) the
deformations, or liftings, of $\Ac(V)$; that is, Hopf algebras $A$
such that $\gr A$ (the graded Hopf algebra with respect the
coradical filtration) is isomorphic to $\Ac(V)$.

\item\label{item:lefting-iii}
If $A$ is a \fd{} Hopf algebra with $A_0\simeq H$, then prove that there exists $V$ such that
$\gr A\simeq \Ac(V)$.
\end{enumerate}

\subsection{}\label{subsec:intro3} 
Two fusion categories $\C$ and $\D$ are \emph{Morita-equivalent}
if there exists an indecomposable $\C$-module category $\M$
such that $\D$ is tensor equivalent to $\End_{\C} (\M)$ \cite{Mu1};
equivalently, if their centers are equivalent as braided tensor categories \cite[Theorem 3.1]{ENO2}. In this case we write $\C \mor \D$.

\renewcommand{\thefootnote}{\alph{footnote}}

Two semisimple Hopf algebras $K$ and $H$ are \emph{Morita-equivalent} 
(denoted by $K \mor H$)
 iff  $\Rep K \mor \Rep H$ \footnote{This is not the same as being Morita-equivalent as algebras.}, iff $\ydk$ and $\ydhh$  are equivalent as braided tensor categories.
When this is the case, the braided equivalence $\Fc:\ydk \to \ydhh$ preserves Nichols algebras, i.e. $\Fc(\toba (V))
\simeq \toba (\Fc(V))$. 
In consequence, if Problem \eqref{item:lefting-i} above is solved for $K$, then so is for $H$.
Also, Problem \eqref{item:lefting-iii} is equivalent to: if $\toba \twoheadrightarrow \toba(V)$ is a \fd{} pre-Nichols algebra in $\ydk$, then necessarily $\toba \simeq \toba(V)$. Therefore, if Problem \eqref{item:lefting-iii} above is solved for $K$, then so is for $H$.

In this paper, we construct Hopf algebras with the Chevalley property over a semisimple Hopf algebra $H$ that is Morita-equivalent to a group algebra $K =\ku G$, $G$ a finite group, provided we know examples, or even better the classification, of $V \in \ydk$ with $\dim \toba(V) < \infty$.

\subsection{}\label{subsec:intro4}

Let $G$ be a finite group. The characterization of all semisimple Hopf algebras Morita-equivalent to   $\ku G$ follows from \cite{O} as we briefly recall now. 

\smallbreak \noindent\emph{$\circ$ Duals.}
If $H$ is a semisimple Hopf algebra, then 
$H^* \mor H$-- just take $\M = \vect{_\ku}$.
In particular $\vect_{G} \mor \Rep \ku G$, where $\vect_{G}$ is the category  of \fd{} $G$-graded vector spaces.

\smallbreak \noindent\emph{$\circ$ Group-theoretical Hopf algebras.} Let  $F,  \Gamma< G$ be such that $G =F \Gamma$--but $F \cap \Gamma$ need not be trivial. Given a suitable pair $(\alpha, \beta) \in H^2(F, \ku^{\times}) \times H^2(\Gamma, \ku^{\times})$, cf. Definition \ref{def:gp-th}, there is a corresponding Hopf algebra $\gth{G}{F}{\Gamma}{\alpha}{\beta}$ such that $\gth{G}{F}{\Gamma}{\alpha}{\beta} \mor  \ku G$. The collection $(F,\alpha, \Gamma, \beta)$
is called a  \emph{group-theoretical datum} for $G$.
 These are all Hopf algebras arising from fiber functors of all  fusion categories Morita-equivalent to $\vect_{G}$ \cite{O}, therefore all $H\mor \ku G$ are like this. See \S \ref{subsec:Group-theoretical} for more details.
 Notice however that to decide when two of these Hopf algebras are isomorphic is not evident.

We describe next some particular instances of this notion. 

\smallbreak \noindent\emph{$\circ$ Twistings.} 
Two finite-dimensional Hopf algebras $H$ and $H'$ are
twist-equivalent if and only if  $\Rep H$ and $\Rep
H'$ are tensor equivalent \cite{EG, S}.
Therefore, if  $J \in \ku G \otimes \ku G$ is a twist, then $(\ku G)^J\mor \ku G$.
Furthermore, if $U$ is a pointed Hopf algebra with $G(U) \simeq G$, then $J \in U \otimes U$ is a twist and the Hopf algebra $U^J$, with coradical $(\ku G)^J$, has the dual Chevalley property. Hence, if Problem \eqref{item:lefting-ii} is solved for $\ku G$, then so is for $(\ku G)^J$.

\smallbreak \noindent\emph{$\circ$ Abelian extensions \cite{K,Ta}.} Assume that  $G = F \Gamma$ is an exact factorization (i.e., $F, \Gamma< G$ with $G =F \Gamma$,  $F \cap \Gamma = 1$) and that $H$ fits into an exact sequence 
\begin{align}\label{eq:exact-seq}
\xymatrix{1\ar[r] & \ku^{\Gamma} \ar^\iota@{^{(}->}[r] & H
	\ar^\pi@{->>}[r] & \ku F \ar[r] & 1},
\end{align}
 associated to this factorization. 
Let $(\sigma, \tau) \in \opext (\ku^{\Gamma}, \ku F)$ be the corresponding pair of  2-cocycles. Thus $H = \ku^{\Gamma} {\,}^{\tau}\hspace{-0.08cm}\#_{\sigma} \ku F$ is a bicrossed product.
Let $\omega\in Z^3(G, \ku^{\times})$ be the 3-cocycle arising from $(\sigma, \tau)$ in the Kac exact sequence; we say that $H$ has 3-cocycle $\omega$.
  Albeit $\Rep \ku G$ and $\Rep H$ are not necessarily tensor-equivalent,  $D(H)$ is isomorphic to a twist of $D^{\omega}(\ku G)$ \cite[Theorem 1.3]{Natale}.  Hence, if $\omega = 1$, then  $H\mor \ku G$ by \cite[Theorem 3.1]{ENO2}.

\subsection{}\label{subsec:intro5}
As is well-known, the Nichols algebra $\toba(V)$ depends essentially
only on the underlying braided vector space to the Yetter-Drinfeld
module $V$. Here we shall not consider braided vector spaces of diagonal type--except for Yetter-Drinfeld modules over some dihedral groups, see Table \ref{tab:gp-th-data-dn-even} in Section \ref{sec:dn}. We focus on braided vector spaces of rack type, see e.g. \cite{AG-adv} or  \cite{AFGV-ampa}. For more examples with braided vector spaces of diagonal type see \cite{de1tipo6chevalley, Mom}.
Let  $(X, \q)$ be a pair where $X$ is a rack and $q$ a 2-cocycle, let $V$ be the associated braided vector space and assume that  $\toba(V)$ is \fd{}, cf. \cite{HLV}. We consider a group $G$ such that  $V$ is realized in $\ydG$. We then compute all group-theoretical data
 $(F,\alpha, \Gamma, \beta)$ for $G$. Consequently, $H = \gth{G}{F}{\Gamma}{\alpha}{\beta} \mor \ku G$ and there is $V'\in \ydhh$ such that $\toba(V') \simeq \toba(V)$, as algebras and coalgebras. 
We summarize our computations in Table \ref{tab:app2}.
\begin{table}[h]
	\begin{center}
		\caption{}\label{tab:app2}
		
		\begin{tabular}{|c|c|c|c|c|}
			\hline     $(X, \q)$ &   $\dim \toba(V)$ & {\small {\bf Reference }} &  $G$ & $H \mor \ku G$ 
			\\\hline $\D_3$, -1       & $12$ & \cite{MS}  & $C_3\rtimes C_6$ &  Prop. \ref{prop:c3c6}
			
			\\ \hline $(\Q_{5, 2},-1)$, $(\Q_{5, 3},-1)$
			& $1280$ & \cite{AG-adv} & $C_5\rtimes_2 C_{20}$  	& Table \ref{tab:gp-th-data-cccv}

			\\ \hline {\small $(\Oc^4_2,   -1)$, $(\Oc^4_2, \chi)$ $(\Oc^4_4, -1)$}
			& $576$ & {\small\cite{FK, MS}} & $\mathbb{S}_4$  & Table \ref{tab:gp-th-data-sk}

			\\ \hline $(\Oc^5_2,-1)$, $(\Oc^5_2, \chi)$
			& $8294400$ & {\small\cite{FK, G-zoo,GG}} & $\mathbb{S}_5$  &  
			Table \ref{tab:gp-th-data-s5}

			\\ 
	\hline $(\T,-1)$ 
			& $72$ & {\small\cite{G-cm}} &  $\mathbb{A}_4\times C_2$ &   
			Table \ref{tab:gp-th-data-a4c2}

			 \\ \hline $(\Q_{7, 3},-1)$,  $(\Q_{7, 5},-1)$
            & $326592$ & \cite{G-zoo}  &  $C_7\rtimes_3 C_6$        
            &    Prop. \ref{prop:Afin(7,5)}

			\\ \hline
		\end{tabular}
		
	\end{center}
\end{table}

We describe new examples in Theorems \ref{prop:c3c6-pointed}, \ref{prop:c5c20-pointed},  \ref{prop:S4-pointed}, \ref{prop:S5-pointed}, \ref{prop:A4xZ2-pointed}, \ref{prop:c7c6-pointed}, \ref{prop:dn-even-pointed}.

\begin{obs}
 (i)
 There are groups $G$ that admit a finite-dimensional Nichols algebra but no non-trivial $H$ with $H \mor \ku G$. For instance $(\D_3,-1)$ corresponds to some $V \in {}^{\ku \st}_{\ku \st}  \mathcal{YD}$
but $H \mor \ku \st$ implies $H \simeq \ku \st $ or $\ku^{\st}$, \S  \ref{sec:c3c6}. Also, $(\Q_{5, 2},-1)$
corresponds to some  $V \in \ydG$, where $G = \ku (C_5\rtimes_2 C_4)$;
but $H \mor \ku G$ implies $H \simeq \ku G$ or $\ku^{G}$, \S \ref{sec:c5c20}.

(ii) Let $G$ be as in Table \ref{tab:app2} or a dihedral group $D_n$.
 By Lemma \ref{prop:algebras gt equivalentes} \ref{item:cocommutative}, there is no group $G'\not\simeq  G$ with $\ku G'\mor \ku G$. 
 
 (iii). If  $J \in \ku G \otimes \ku G$ is a twist and $U$ is a pointed Hopf algebra with $G(U) \simeq G$, then  the Hopf algebra $U^J$, with coradical $(\ku G)^J$, has a rather concrete description. Otherwise,
 if $H \mor \ku G$ , then the braided equivalence $\Fc:\ydG \to \ydhh$
 is not quite explicit. In this way, neither the structure of the Yetter-Drinfeld module $\Fc(V)$ nor the defining relations of $\toba(\Fc(V))$ are not known, and the description of the liftings is problematic. 
Notice that a direct relation between liftings of $\toba(V)\# \ku G$ and liftings of  $\toba(\Fc(V))\# H$ is not expected. For instance, if $G = \s_3$ and $\toba(V)$ is the 12-dimensional Nichols algebra, then 
 $\toba(V)\# \ku \s_3$ has exactly one non-trivial lifting \cite{AG2}; whereas 
 $\toba(V)\# \ku^{\s_3}$ has infinitely many non-isomorphic liftings 
 \cite{AV}.
 \end{obs}
 
 We use \gap  \ (see \cite{GAP4}) for some of the computations with finite groups. For most of the computations of finite group cohomology, we use the \gap~pa\-ckage \anupq~and  the natural isomorphism $H^{n}(\underline{\ },\ku^{\times})\cong  H^{n+1}(\underline{\ },\mathbb{Z})$.

\smallbreak\subsubsection*{Acknowledgements} We thank Graham Ellis for helping us to use the \gap~pa\-ckage \anupq, Leandro Vendramin and Cristian Vay for kindly answering some questions, and Agust\'in Garc\'ia Iglesias for very useful conversations.

\subsection{Preliminaries and notations}\label{subsec:preliminaries}

As customary we refer to group algebras and their duals as trivial Hopf algebras. We denote by $1$ the identity element of a group.
If $n \in \N$, then $C_n$ is a cyclic group of order $n$. If $G$ is a group, then the notation $F<G$ means that $F$ is a subgroup of $G$, while $F \lhd  G$ means that $F< G$ is normal. If $\triangleright$ is an action of a group $G$ on a set $F$, then we denote by $F^G$ the subset of $F$ of points fixed by $\triangleright$.  The dihedral group of order $2n$ is denoted by $D_n$.
We  denote multiplicatively the cohomology groups $H^n(G,\ku^{\times})$. Occasionally, we denote by the same letter an element in $H^n(G,\ku^{\times})$ and any of its representatives. Let $a, b \in \mathbb{Z}$, the lowest common multiple of $a$ and $b$ is denoted by $[a,b]$ and the greatest common divisor of $a$ and $b$ is denoted by $(a,b)$.  We denote by $\G_n=\{z\in \ku^\times: z^n=1\}$, $n \in \N$, and by $\G_n'$ the primitive ones.

One says that $\alpha \in Z^2(G,\ku^\times)$ is non-degenerate if and only if  the twisted group algebra $\ku_\alpha G$ is simple; see e.g. \cite{Dav, Mov}. Note that this definition only depends on the cohomology class of $\alpha$.

If $A$ is an abelian group and
$T\in\Aut A$, then $A$ is a rack (called affine) with $x\trid y=(\id-T)x+Ty$, $x, y \in A$. This is  denoted by $\Q_{A,T}$; or $\Q_{q,b}$ when $A \simeq \mathbb{F}_q$, where $q$ is a prime power, and  $T \in \Aut \mathbb{F}_q$, $T(x)= b x$, $x \in \F_q$, $b \in \F_q^{\times} \simeq C_{q-1}$. Suppose that the order of $b$ divides $q-1$. Then the rack $\Q_{q,b}$ can be realized as a conjugacy class of the group $C_q\rtimes_b C_{q-1}$, where the subscript $b$ describes the action.
Another exception to the notation is $\D_n = \Q_{C_n,T}$, where $T(x) = - x$, $x \in C_n$ (the so-called dihedral rack). Also $\T = \Q_{4,b}$, $b \in \mathbb{F}_4$  irreducible, is called the tetrahedral rack.
Meanwhile, $\Oc^n_j$ is the conjugacy class of $j$-cycles in $\sn$.

\section{Semisimple Hopf algebras}\label{sec:ss-Hopfalg}
Let $G$ be a finite group.

\subsection{Twists}\label{subsec:twistings-of-groups}

Twists in $\ku G$  are classified, up to gauge equivalence, 
by conjugacy classes of pairs $(S,\alpha)$, where $S < G$ and $\alpha \in H^2(S,\ku^\times)$ is non-degenerate \cite{Mov, EG3}. Hence $S$ is solvable and $\vert S\vert$ is a square.

\subsection{Group-theoretical Hopf algebras}\label{subsec:Group-theoretical}

Let $\omega\in Z^3(G, \ku^{\times})$.
The tensor category of finite-dimensional $G$-graded vector spaces
with associativity constraint given by $\omega$ is denoted by
$\vect_G^{\omega}$.  The tensor categories $\vect_G^{\omega}$ and
$\vect_G^{\nu}$ are tensor equivalent iff $f^*(\overline{\omega}) =
\overline{\nu}$ in $H^3(G, \ku^{\times})$ for some $f \in \Aut G$.
If   $F < G$ and $\alpha\in C^2(F, \ku^{\times})$
such that $d \alpha = \omega_{\vert F\times F \times F}$, then the twisted
group algebra $\ku_{\alpha} F$  is 
associative  in $\vect_G^{\omega}$. The category $\C(G, \omega; F, \alpha)$ of $\ku_{\alpha} F$-bimodules in
$\vect_G^{\omega}$ is a fusion category and
$\C(G,\omega; F, \alpha) \mor \vect_G^{\omega}$. 
By Tannaka duality, to describe the semisimple Hopf algebras $H$ such that $\Corep H$
is tensor equivalent to $\C(G, \omega; F, \alpha)$ is tantamount
to describe the fiber functors from $\C(G, \omega; F, \alpha)$ to
$\vect_{\ku}$, that is contained in the more general
description of indecomposable module categories over $\C(G,\omega; F, \alpha)$ \cite[Corollary 3.1]{O}.  
These Hopf algebras are called group-theoretical.
By \cite{S1}, the centers of $\C(G,\omega; F, \alpha)$ and  $\vect_G^{\omega}$ are equivalent as braided tensor categories. 

We are interested in group-theoretical Hopf algebras $H\mor \ku G$, i.e. with $\omega = 1$. Concretely, we introduce the following terminology.

\begin{definition}\label{def:gp-th} A \emph{group-theoretical datum} for $G$
is a collection $(F,\alpha, \Gamma, \beta)$ where $F, \Gamma < G$, $\alpha\in H^2(F, \ku^{\times})$ and  $\beta\in H^2(\Gamma,
\ku^{\times})$, satisfy
\begin{itemize}[leftmargin=*]\renewcommand{\labelitemi}{$\circ$}
    \item  $G = F\Gamma$;
    \item  $\alpha_{\vert F\cap\Gamma}\cdot{\beta_{\vert F\cap\Gamma}}^{-1}$
    is  non-degenerate in $F\cap\Gamma$.
\end{itemize}
\end{definition}
Let $(F,\alpha, \Gamma, \beta)$ be a group-theoretical datum  for $G$ and choose representatives $\underline{\alpha}$, $\underline{\beta}$. By \cite{O}, 
there is a Hopf algebra $H = \gth{G}{F}{\Gamma}{\underline{\alpha}}{\underline{\beta}}$ with
$\Corep H$ tensor equivalent to $\C(G, 1; F, \alpha)$, and a fortiori $H\mor \ku G$; we say that $H$ is a group-theoretical Hopf algebra over $G$. Up to isomorphisms, $H$ only depends on $\alpha$ and $\beta$;
so we loosely denote $H = \gth{G}{F}{\Gamma}{\alpha}{\beta}$.
Conversely, every Hopf algebra $H$ with $H\mor \ku G$ is isomorphic to $\gth{G}{F}{\Gamma}{\alpha}{\beta}$
for some group-theoretical datum $(F,\alpha, \Gamma, \beta)$ for $G$ by \cite{O}, as sketched above.

Let $\theta\in \Aut G$ and $\gamma \in H^2(G, \ku^\times)$. If  $(F,\alpha, \Gamma, \beta)$ is a group-theoretical datum for  $G$,  then $(\theta(F), \theta_*(\alpha\gamma_{\vert F}), \theta(\Gamma), \theta_*(\beta\gamma_{\vert \Gamma}))$ is again one. Here $\theta_*(\alpha\gamma_{\vert F})$ is the pushforward cocycle, i.e. $\theta_*(\alpha\gamma_{\vert F})(\theta(a),\theta(b))$ $=\alpha(a,b)\gamma(a,b)$, for  $a,b\in F$ and analogously for $\theta_*(\beta\gamma_{\vert \Gamma})$.
Thus  $\Aut G \ltimes H^2(G, \ku^\times)$ acts on the set of group-theoretical data; we say that two group-theoretical data
are \emph{equivalent} if they belong to the same orbit of this action.

\begin{lema}\label{prop:algebras gt equivalentes} Let $(F,\alpha, \Gamma, \beta)$ be a group-theoretical datum for $G$.
\begin{enumerate}[leftmargin=*,label=\rm{(\roman*)}] 
	\item\label{item:iso}  If $\theta\in \Aut G$ and $\gamma \in H^2(G, \ku^\times)$,  then as Hopf
	algebras
	\begin{align*}
\gth{G}{F}{\Gamma}{\alpha}{\beta} \simeq \gth{G}{\theta(F)}{\theta(\Gamma)}{\theta_*(\alpha\gamma_{\vert F})}{\theta_*(\beta\gamma_{\vert \Gamma})}.
	\end{align*}
\hspace{-35pt}	That is, equivalent group-theoretical data produce isomorphic Hopf algebras.

\item\label{item:cocommutative} $\gth{G}{F}{\Gamma}{\alpha}{\beta}$ is cocommutative (resp. commutative) iff $F \lhd G$ is abelian   and $\alpha \in H^2(F, \ku^{\times})^{\ad G}$ (resp. $\Gamma  \lhd G$ is  abelian and
$\beta \in H^2(\Gamma, \ku^{\times})^{\ad G}$).   If $H=\gth{G}{F}{\Gamma}{\alpha}{\beta}$ is cocommutative then $\ku G(H) \mor \ku G$ and every  $G'$ such that $\ku G' \mor \ku G$ arises in this way.

\item\label{item:duals} $\gth{G}{F}{\Gamma}{\alpha}{\beta}^* \simeq\gth{G}{\Gamma}{F}{\beta}{\alpha}$.

\item\label{item:twist}  $\gth{G}{F}{G}{\alpha}{1}$ is a twisting of $\ku G$ and every twisting is of this form.

\item\label{item:abext} If $F\cap \Gamma = 1$, then  $\gth{G}{F}{\Gamma}{\alpha}{\beta}$ is an abelian extension of $\ku^F$ by $\ku\Gamma$: 
\begin{align*}
\ku^F \hookrightarrow \gth{G}{F}{\Gamma}{\alpha}{\beta} \twoheadrightarrow \ku \Gamma.
\end{align*}

\end{enumerate}
\end{lema}
\pf \ref{item:iso}  Let $H=\gth{G}{F}{\Gamma}{\alpha}{\beta}$, $\mathcal{C}=\vect_G, A=\ku_\alpha F$ and $B=\ku_\beta \Gamma$.    Clearly $H$ is the Tannakian reconstruction of the fiber functor   $$_A\mathcal{C}_A \to \End(_A\mathcal{C}_B)\simeq_\otimes \vect_{\ku}.$$  
If $(\theta, \gamma) \in \Aut G \ltimes Z^2(G, \ku^\times)$, then it induces a tensor automorphism $(\theta_*, \gamma): \vect_G \to \vect_G$ by
\begin{align*}
\ku_g &\mapsto \ku_{\theta(g)},& \gamma_{\ku_g, \ku_h}:=\gamma(g,h)\id_{\ku_{\theta(gh)}}:\ku_{\theta(g)}\otimes \ku_{\theta(h)}&\to \ku_{\theta(gh)}.
\end{align*}
The object $\theta_*(A)=\bigoplus_{g\in F} \ku_{\theta(g)}$ is an algebra in $\C$ with multiplication $u_{\theta(g)}u_{\theta(h)}=\gamma(g,h)\theta(g,h)u_{\theta(gh)}$. The
tensor automorphism $(\theta_*, \gamma)$ induces  a tensor equivalence $(\theta_*, \gamma):\ _{A}\C_{A}\to \ _{\theta_*(A)}\C_{\theta_*(A)}$ such that the following diagram of tensor functors commutes:

\begin{equation}\label{diagrama 1}
\xymatrix{
{_A\C_A}\ar[d]^{\theta_*}  \ar[r] & {\End(_A\mathcal{C}_B)} \ar[d]_{\operatorname{Id}}\\
{_{\theta_*(A)}\C_{\theta_*(A)}} \ar[r] &  {\End(_{\theta_*(A)}\mathcal{C}_{\theta(B)})}. }
\end{equation}
 The Tannakian reconstruction of  $_{\theta(A)}\C_{\theta(A)}\to \End(_{\theta(A)}\mathcal{C}_{\theta(B)})$ is $$\gth{G}{\theta(F)}{\theta(\Gamma)}{\theta_*(\alpha\gamma_{\vert F})}{\theta_*(\beta\gamma_{\vert \Gamma})},$$ so by Tannakian formalism and the commutativity of \eqref{diagrama 1}, the Hopf algebras $\gth{G}{F}{\Gamma}{\alpha}{\beta}$ and $\gth{G}{\theta(F)}{\theta(\Gamma)}{\theta_*(\alpha\gamma_{\vert F})}{\theta_*(\beta\gamma_{\vert \Gamma})}$ are isomorphic.

\ref{item:cocommutative} A semisimple Hopf algebra $H$ is cocommutative, respectively commutative, if and only if $\Corep H$, respectively $\Rep H$, is pointed.  By \cite[Theorem 3.4]{Naidu},  $\C(G, 1; F, \alpha)$ is pointed if and only if $F$ is a normal abelian subgroup of $G$ and $\alpha$ is  $\ad G$-invariant. This implies the claim.

\ref{item:duals} Let $(F,\alpha,\Gamma,\beta)$ be a group-theoretical datum for $G$. We have associated  two fusion categories $ \  _{\ku_\alpha F}(\vect_G)_{\ku_\alpha F}$, $\  _{\ku_\beta \Gamma}(\vect_G)_{\ku_\beta \Gamma}$ and the bimodule category of rank one $\ _{\ku_\alpha F}(\vect_G)_{\ku_\beta \Gamma}$. The Hopf algebra $\gth{G}{F}{\Gamma}{\alpha}{\beta}$ (resp. $\gth{G}{\Gamma}{F}{\beta}{\alpha}$) is by definition the Tannaka reconstruction of $ \  _{\ku_\alpha F}(\vect_G)_{\ku_\alpha F}$ (resp. $\  _{\ku_\beta \Gamma}(\vect_G)_{\ku_\beta \Gamma}$) respect to the fiber functor defined by the left (resp. right) module category $\ _{\ku_\alpha F}(\vect_G)_{\ku_\beta \Gamma}$. Now, it follows by \cite[Appendix C]{Ya} that $\gth{G}{F}{\Gamma}{\alpha}{\beta}^* \simeq\gth{G}{\Gamma}{F}{\beta}{\alpha}$.

\ref{item:twist} The fusion category  $$\Rep \gth{G}{F}{G}{\alpha}{1} \simeq
\Corep \gth{G}{G}{F}{1}{\alpha} \simeq \C(G, 1; G, 1)= {}_{\ku G}(\vect_{G})_{\ku G}$$ is tensor equivalent to $\Rep G$. Then $\gth{G}{F}{G}{\alpha}{1}$ is twist equivalent to $\ku G$. Conversely, if  $H\simeq (\ku G)^J$, then $\Rep H\simeq \Rep G\simeq 
\C(G, 1; G, 1)$. Thus $H\simeq \gth{G}{F}{G}{\alpha}{1}$ for some $F<G$ and $\alpha\in Z^2(F, \mathbb{C}^\times)$ non-degenerate. 

\ref{item:abext} Let $A$ be any extension of $\ku^{F}$ by  $\ku \Gamma$ with 3-cocycle $\omega$. By \cite[Lemma 6.3.1]{S-adv},
 $A^{J(\beta)}$ and $A_\alpha$ are extensions of $\ku^{F}$ by  $\ku \Gamma$ with 3-cocycle $\omega$. 
 Now by \cite[Remark 4.6]{Natale}, $H = \gtha{G}{\omega}{F}{\Gamma}{\alpha}{\beta}\simeq A^{J(\beta)}_\alpha = (A^{J(\beta)})_\alpha$; notice 
 that $A^G_{\alpha, \beta}(\omega, \Gamma, F) \simeq H^*$, because of  the convention in \cite{Natale} that $\Rep A^G_{\alpha, \beta}(\omega, \Gamma, F)$
is tensor equivalent to $\C(G, \omega; F, \alpha)$.
Hence $H$ is an extension of $\ku^{F}$ by  $\ku \Gamma$ with 3-cocycle $\omega$, applying twice \cite[Lemma 6.3.1]{S-adv}.
Since the split extension $\ku^{F} \# \ku \Gamma$ has 3-cocycle $1$, $\gth{G}{F}{\Gamma}{\alpha}{\beta}$ turns out to be an abelian extension.
\epf

We shall compute all group-theoretical data of some specific groups and then determine the isomorphism classes of the 
corresponding group-theoreti\-cal Hopf algebras. We present now some auxiliary results for this end.

\begin{obs}\label{obs:aux-opext} Let $(F,\alpha, \Gamma, \beta)$ be a group-theoretical datum for $G$. If $\gamma \in H^2(G, \ku^\times)$, such that $\gamma\vert_F= \alpha^{-1}$ then $(F,\alpha, \Gamma, \beta)$ and  $(F,1, \Gamma, \beta\gamma\vert_\Gamma)$ are equivalent, so $\gth{G}{F}{\Gamma}{\alpha}{\beta} \simeq  \gth{G}{F}{\Gamma}{1}{\beta\gamma\vert_\Gamma}$. Similarly, with $\Gamma$ instead of $F$.
\end{obs}

\begin{lema}\label{lema:aux-G(H)} Let $(F,\alpha, \Gamma, \beta)$ be a group-theoretical datum, $H = \gth{G}{F}{\Gamma}{\alpha}{\beta}$, $K= \{\overline{g} \in N_G(F)/F: [\alpha] = [\alpha^g] \}$. Then $G(H)$ fits into an exact sequence
\begin{align}\label{eq:G(gp-th)}
\xymatrix{1\ar[r] & \widehat{F} \ar@{^{(}->}[r] & G(H)
	\ar@{->>}[r] & K \ar[r] & 1},
\end{align}
\end{lema}

\pf Since $G(H)$ is isomorphic to the group of invertible objects in the tensor category $\Corep H \simeq \C(G, 1; F, \alpha)$,
the claim follows from \cite{GeN}.
\epf

The action and the cocycle in \eqref{eq:G(gp-th)} are explicitly given in \cite{GeN}. In particular, if $\alpha = 1$, then  $G(H) \simeq \widehat{F} \rtimes  N_G(F)/F$.

\begin{lema}\label{lema:aux-twist} Let $G'$ be a finite group. Let $(F,\alpha, \Gamma, \beta)$ and $(F',\alpha', \Gamma', \beta')$ be group-theoretical data for $G$ and $G'$ respectively. Then  $\Rep \gth{G}{F}{\Gamma}{\alpha}{\beta}$ and $\Rep \gth{G'}{F'}{\Gamma'}{\alpha'}{\beta'}$ are equivalent as tensor categories if and only if there is an invertible 
$\vect_G$-$\vect_{G'}$-bimodule category $\mathcal{X}$ such that 
$$\mathcal{X}\boxtimes_{\vect_{G'}}\M(\Gamma',\beta')\simeq \M(\Gamma,\beta)$$ as $\vect_G$-module categories, where $\M(\Gamma',\beta')$ and $\M(\Gamma,\beta)$ are the module categories associated to the pairs $(\Gamma',\beta')$ and $(\Gamma,\beta)$.
\end{lema}

\pf Recall that $\Rep \gth{G}{F}{\Gamma}{\alpha}{\beta}\simeq_\otimes \C(G,1;\Gamma, \beta)$ and $\Rep \gth{G'}{F'}{\Gamma'}{\alpha'}{\beta'}\simeq_\otimes \C(G',1;\Gamma', \beta')$. Then \cite[Proposition 5.1]{GJ} applies.
\epf

\section{Group-theoretical Hopf algebras over $G = C_3\rtimes C_6$}\label{sec:c3c6} 

The group $\st$ has a Yetter-Drinfeld module $V$ with $\dim \toba(V) = 12$ \cite{MS},
but there is no non-trivial group-theoretical Hopf algebra over $\st$. 
Indeed, the classification of the Hopf algebras of dimension $6 = |\mathbb{S}_3|$ is known:  $\ku C_6$, $\ku\mathbb{S}_3$ and $\ku^{\mathbb{S}_3}$.
Hence the only group-theoretical Hopf algebras over $\mathbb{S}_3$, up to Morita-equivalence, are $\ku\mathbb{S}_3$ and $\ku^{\mathbb{S}_3}$.

However the braided vector space $V$ can be realized as a Yetter-Drinfeld module over $G_m=C_3\rtimes C_{2m}$, $m\in \mathbb{N}$, see \cite[\S 4]{GVay}. Here we deal with the group-theoretical Hopf algebras over  $C_3\rtimes C_6$.  Notice that the classification of the semisimple Hopf algebras of dimension 18 is known \cite{Masuoka}.

Let $\xi\in \G_3$, $L=C_2=\langle x\rangle$, $N = C_3\times C_3=\langle a\rangle \times \langle b\rangle$. The Hopf algebra $A_{18,\xi}$, defined in \cite[1.2]{Masuoka}, is the abelian extension associated to the matched pair $(L, N)$ with $\triangleright: N \times L\to L$ trivial, $\triangleleft : N \times L\to N$ given by $a^ib^j \triangleleft x=a^ib^{-j}$ and cocycles $\sigma : L\times L\to (\ku^N)^\times$ trivial and $\tau : N \times N \to (\ku^L)^\times$ given by $\tau(a^ib^j, a^rb^s)=\delta_1+\xi^{jr}\delta_x$. Accordingly, $G =  C_3\rtimes C_6 = \langle x, a, b\rangle$.
Also, $A_{18,\xi} \simeq A_{18,\eta} \iff \vert \xi\vert = \vert \eta\vert$ \cite[1.5]{Masuoka}.

\begin{prop}\label{prop:c3c6} The non-trivial group-theoretical Hopf algebras over $G$ are $A_{18, \xi}$ and $(A_{18, \xi})^*$, $\xi\in \G'_3$.
\end{prop}

\pf Since  $|G|=2\times 3^2$, the unique non-trivial subgroup with a non-degenerated 2-cocycle is    $N \simeq C_3\times C_3$. Let $M =\langle x,b\rangle \simeq \st$. Let $(F,\alpha, \Gamma, \beta)$ be a group-theoretical data for $G$; then $F \cap \Gamma$ is either 1 or $N$.

\noindent\emph{Case 1. } $F \cap \Gamma = 1$.  Up to conjugation, either $(F, \Gamma)$ or $(\Gamma, F)$ is one of

\begin{minipage}{.3\textwidth}
\begin{align}
\label{c3c6-uno}
 &(L, N)
\\ \label{c3c6-dos}
 & (\langle b \rangle, \langle x, a \rangle)
 \end{align}
 \end{minipage}
 \begin{minipage}{.3\textwidth}
 \begin{align}
\label{c3c6-tres}
 & (\langle a \rangle, M)
\\ \label{c3c6-cuatro}
 & (\langle ab \rangle, M)
 \end{align}
 \end{minipage}
 \begin{minipage}{.3\textwidth}
\begin{align}
 \label{c3c6-cinco}
& (\langle ab \rangle, \langle x, a \rangle).
\end{align}
\end{minipage}

If $(F, \Gamma)$ is as in \eqref{c3c6-dos}, \eqref{c3c6-tres}, $F\lhd G$ hence $\gth{G}{F}{\Gamma}{\alpha}{\beta}$  is cocommutative, cf. Lemma \ref{prop:algebras gt equivalentes} \ref{item:cocommutative}. If $(F, \Gamma)$ is as in \eqref{c3c6-cuatro}, $F \ntriangleleft G$ and $\Gamma$ is non-abelian, hence $\gth{G}{F}{\Gamma}{\alpha}{\beta}$ is non-trivial. If $(F, \Gamma)$ is as in \eqref{c3c6-cinco}, $F, \Gamma\ntriangleleft G$, hence $\gth{G}{F}{\Gamma}{\alpha}{\beta}$ is non-trivial. If $(F, \Gamma)$ is as in \eqref{c3c6-uno}, $F\ntriangleleft G$, $H^2(N, \ku^\times )\simeq C_3$ and $H^2(N, \ku^\times)^{\ad G}=1$, then $\gth{G}{G}{N}{1}{\beta}$, $\beta\neq 1$, is non-trivial. Moreover, these cocycles give rise to isomorphic Hopf algebras. 

\noindent \noindent\emph{Case 2. }  $F \cap \Gamma = N$, so that $(F, \Gamma)$  is either $(N, G)$--denoted by (3.6)-- or else  $(G, N)$. Since $H^2(N, \ku^\times)\simeq C_3$ and $H^2(N, \ku^\times)^{\ad G}=1$,  $\gth{G}{G}{N}{1}{\beta}$, $\beta\neq 1$, is non-trivial. Further, these cocycles give rise to isomorphic Hopf algebras. 

So we have the following possibilities:

	\begin{center}
		\begin{tabular}{|c|c|c|c|c|c|c|}
			\hline     $\#$ & $F$ &  $\alpha$ & $\Gamma$ & $\beta$ &  $G(H)$ &  $G(H^*)$
			\\ \hline
			\eqref{c3c6-uno} &  $L\simeq C_2 $ & 1 & $N\simeq C_{3} \times C_3$ & $\neq 1$ &   $C_6$ &  $C_3\times C_3$ 
        \\ \hline
            \eqref{c3c6-cuatro} &  $\langle ab \rangle\simeq C_3 $ & 1 & $M \simeq \mathbb{S}_3$ & 1 &   $C_3\times C_3$ &   $C_6$
            \\ \hline
           \eqref{c3c6-cinco} &  $\langle ab \rangle $ & 1 & $\langle x, a \rangle \simeq C_{6}$ & 1  &   $C_3\times C_3$ &   $C_6$
            \\ \hline
          (3.6) &  $N $ & $\neq 1$ & $G$ & 1 &   $C_3\times C_3$ &   $C_6$
				\\ \hline
		\end{tabular}
	\end{center}

By \cite[2.3, 2.5]{Masuoka}, since $|G(H)|=9$, $\eqref{c3c6-uno}^*\simeq   \eqref{c3c6-cuatro} \simeq  \eqref{c3c6-cinco} \simeq (3.6) \simeq A_{18, \xi}$; since $|G(H)|=6$, $\eqref{c3c6-uno} \simeq  \eqref{c3c6-cuatro}^* \simeq  \eqref{c3c6-cinco}^* \simeq (3.6)^* \simeq (A_{18, \xi})^*$.
\epf

Thus  $A_{18, \xi}\simeq (3.6)$ is a twisting of $\ku G$.

\begin{theorem}\label{prop:c3c6-pointed}	
The Hopf algebras $A_{18, \xi}$ and $(A_{18, \xi})^*$ have a non-zero Yetter-Drinfeld module $V$ with $\dim \toba(V) < \infty$. 
By bosonization, we get new Hopf algebras with the dual Chevalley property of dimension 216. 
\qed
\end{theorem}

The liftings of $\toba(V) \# \ku G$, where $V$ is as above, are classified in \cite[Theorem 6.2]{GVay}. Indeed, let $(x_j)_{0\leqslant j \leqslant 2}$ be the basis of $V$ as in \emph{loc. cit.} Let
\begin{align*}
\Ss= \{(\lambda_1, \lambda_2) \in \ku^2 \text{ satisfying  \cite[(29), (33)]{GVay}}\}.
\end{align*}
For $(\lambda_1, \lambda_2) \in \Ss$,  let $H(\lambda_1,\lambda_2)$ be 
$T(V)\ \# \ku G$ modulo the ideal generated by
\begin{align*}
x^2_0 &- \lambda_1 (1 - g^2_0)& &\text{and}&
x_0x_1 + x_1x_2 + x_2x_0 - \lambda_2(1 - g_0 g_1).    
\end{align*}
Then
\begin{itemize}[leftmargin=*]
\item $H(\lambda_1,\lambda_2)$ is a lifting of 
$\toba(V) \# \ku G$,

\item any lifting of $\toba(V)\ \# \ku G$ is isomorphic to $H(\lambda_1,\lambda_2)$ for some $(\lambda_1, \lambda_2) \in \Ss$,

\item $H(\lambda_1,\lambda_2) \simeq H(\lambda'_1,\lambda'_2)$ iff there exists $\mu \in \ku^{\times}$ such that $(\lambda_1,\lambda_2) = \mu (\lambda'_1,\lambda'_2)$.
\end{itemize}
\begin{obs}\label{obs-lifting-c3c6} The classification of all liftings of $\toba(V) \# A_{18, \xi}$ follows from \cite[Th. 6.2]{GVay}. Namely, let $J \in \ku G \otimes \ku G$ such that
$A_{18, \xi}\simeq (\ku G)^J$. Then
\begin{itemize}[leftmargin=*]
\item $H(\lambda_1,\lambda_2)^J$ is a lifting of 
$\toba(V)\ \# A_{18, \xi}$, for every $(\lambda_1, \lambda_2) \in \Ss$,

\item any lifting of $\toba(V)\ \# A_{18, \xi}$ is $\simeq$ to $H(\lambda_1,\lambda_2)^J$ for some $(\lambda_1, \lambda_2) \in \Ss$,

\item $H(\lambda_1,\lambda_2)^J \simeq H(\lambda'_1,\lambda'_2)^J$ iff there is $\mu \in \ku^{\times}$ such that $(\lambda_1,\lambda_2) = \mu (\lambda'_1,\lambda'_2)$.
\end{itemize}
\end{obs}

\section{Group-theoretical Hopf algebras over $G = C_5\rtimes_2 C_{20}$}\label{sec:c5c20} 

The group $G = C_5\rtimes_2 C_4$ has two Yetter-Drinfeld modules $V_j$, $j = 2,3$, with $\dim \toba(V_j) = 1280$ \cite{AG-adv}; $V_j$ has braided vector space with rack $\Q_{5, j}$ and cocycle $-1$, and $V_2 \simeq V_3^*$ in $\ydG$. 
But a group-theoretical Hopf algebra  over $C_5\rtimes_2 C_4\simeq C_5\rtimes_3 C_4$ is trivial. 
Indeed,   a non-trivial subgroup of $C_5\rtimes_2 C_4$ does not admit a non-degenerated 2-cocycle,  (alternatively,  there is no triangular Hopf algebra of dimension 20 \cite{Gelaki, Natale-pq2}). Thus such  a group-theoretical Hopf algebra would be an abelian extension, hence trivial.

However the braided vector spaces $V_j$, $j =2$ or 3, can be realized as  Yetter-Drinfeld modules over $C_5\rtimes_2 C_{4m} \simeq C_5\rtimes_3 C_{4m} $, $m\in \mathbb{N}$ \cite[\S 4]{GVay}. Here we deal with group-theoretical Hopf algebras over  $C_5\rtimes_2 C_{20}$. 

\begin{prop} The non-trivial group-theoretical Hopf algebras over $G = C_5\rtimes_2 C_{20}$  are given by the group-theoretical data in Table \ref{tab:gp-th-data-cccv}.
\end{prop}

\begin{question}
Is it true that \ref{sec:c5c20}.\ref{cccv-Uno} $\simeq$ \ref{sec:c5c20}.\ref{cccv-Cuatro-dual}
$\simeq$ \ref{sec:c5c20}.\ref{cccv-Cinco-dual} $\simeq$ \ref{sec:c5c20}.\ref{cccv-Twist-dual}?
\end{question}

\newcounter{itmcccv}
\renewcommand{\theitmcccv}{\alph{itmcccv}}
\newcommand{\itemcccv}[1]{\refstepcounter{itmcccv}\ref{sec:c5c20}.\theitmcccv\label{#1}}

\begin{table}[h]
	\begin{center}
		\caption{Group-theoretical data for $C_5\rtimes_2 C_{20}$}\label{tab:gp-th-data-cccv}
		
		\begin{tabular}{|c|c|c|c|c|c|}
			\hline     $\#$ & $F$ &  $\alpha$ & $\Gamma$ & $\beta$ &  $G(H)$
			\\ \hline
			\itemcccv{cccv-Uno} &  $\langle x\rangle\simeq C_4 $ & 1 & $N\simeq C_{5} \times C_5$ & $\neq 1$ & $C_{20}$
				\\ \hline
		\itemcccv{cccv-Uno-dual} &   $N$ & $\neq 1$ & $\langle x\rangle$ & 1 &  $C_5\times C_5$
            \\ \hline
            \itemcccv{cccv-Cuatro} &  $ \langle ab\rangle\simeq C_5 $ & 1 & $\langle b, x\rangle\simeq C_{5} \rtimes C_4$ & 1 & $C_5\times C_5$
				\\ \hline
		\itemcccv{cccv-Cuatro-dual} & $ \langle b, x\rangle$ & 1  & $ \langle ab\rangle$ & 1 &  $C_{20}$
            \\ \hline
            \itemcccv{cccv-Cinco} &  $\langle ab\rangle\simeq C_5 $ & 1 & $\langle x^3a^2\rangle\simeq C_{20}$ & 1 & $C_5\times C_5$
				\\ \hline
		\itemcccv{cccv-Cinco-dual} & $\langle x^3a^2\rangle$ &  1 & $\langle ab\rangle$ & 1 &  $C_{20}$
            \\ \hline
            \itemcccv{cccv-Twist} &  $N\simeq C_5\times C_5 $ & $\neq 1$ &  $G$ & 1 & $C_5\times C_5$
			\\ \hline
		\itemcccv{cccv-Twist-dual} & $G$ &  1 & $N$ & $\neq 1$ &  $C_{20}$
			   \\ \hline
		\end{tabular}
		
	\end{center}
\end{table}

\pf Let  $G =  \langle a,b,x \rangle$ , where $|a| = |b| = 5$, $|x| = 4$, $a$ is central and $xbx^{-1} = b^2$. Then 
$N = \langle a,b \rangle \simeq  C_5\times C_{5}$  is a normal 5-Sylow subgroup. Since $G$ has no subgroup isomorphic to $C_2\times C_2$, the unique non-trivial subgroup with a non-degenerate 2-cocycle is  $N$. 
Let $(F,\alpha, \Gamma, \beta)$ be a group-theoretical data for $G$; then $F \cap \Gamma$ is either 1 or  $N$.

\noindent\emph{Case 1. } $F \cap \Gamma = 1$.  Up to conjugation, either $(F, \Gamma)$ or $(\Gamma, F)$ is   one of

\begin{minipage}{.3\textwidth}
\begin{align}
\label{c5c20-uno}
& (\langle x\rangle, N)
\\ \label{c5c20-dos}
&  (\langle b\rangle, \langle x^3a^2\rangle) 
 \end{align}
 \end{minipage}
 \begin{minipage}{.3\textwidth}
\begin{align}
 \label{c5c20-tres}
 &  ( \langle a\rangle, \langle b, x\rangle) 
\\ \label{c5c20-cuatro}
 &  ( \langle ab\rangle, \langle b, x\rangle) 
\end{align}
\end{minipage}
\begin{minipage}{.3\textwidth}
\begin{align}
\label{c5c20-cinco}
 &  ( \langle ab\rangle, \langle x^3a^2\rangle) 
\end{align}
\end{minipage}

 If $(F, \Gamma)$ is as in \eqref{c5c20-uno}, $F \ntriangleleft G$, $H^2(\Gamma, \ku^\times)^{\ad G}=1$ and $H^2(\Gamma, \ku^\times)\simeq C_5$, then $\gth{G}{F}{\Gamma}{1}{\beta}$, $\beta\neq 1$, is non-trivial. Moreover, these cocycles give rise to isomorphic Hopf algebras, denoted by \ref{sec:c5c20}.\ref{cccv-Uno}. If $(F, \Gamma)$ is as in \eqref{c5c20-dos}, \eqref{c5c20-tres}, $F\lhd G$ hence $\gth{G}{F}{\Gamma}{\alpha}{\beta}$  is cocommutative, cf. Lemma \ref{prop:algebras gt equivalentes} \ref{item:cocommutative}. If $(F, \Gamma)$ is as in \eqref{c5c20-cuatro}, $F \ntriangleleft G$ and $\Gamma$ is non-abelian, then $\gth{G}{F}{\Gamma}{\alpha}{\beta}$ is non-trivial, giving  \ref{sec:c5c20}.\ref{cccv-Cuatro}. If $(F, \Gamma)$ is as in  \eqref{c5c20-cinco}, $F, \Gamma\ntriangleleft G$ hence $\gth{G}{F}{\Gamma}{\alpha}{\beta}$  is non-trivial, thus \ref{sec:c5c20}.\ref{cccv-Cinco}.

  \noindent\emph{Case 2. } $F \cap \Gamma= N$. Up to conjugation, either $(F, \Gamma)$ or $(\Gamma, F)$ is $(N, G)$.  Since $H^2(\Gamma, \ku^\times)^{\ad G}=1$ and $H^2(\Gamma, \ku^\times)\simeq C_5$,  $\gth{G}{F}{\Gamma}{1}{\beta}$, $\beta\neq 1$, is non-trivial. Moreover, these cocycles give isomorphic Hopf algebras, i.e. \ref{sec:c5c20}.\ref{cccv-Twist}. 
\epf

Observe that  \ref{sec:c5c20}.\ref{cccv-Twist}  is a twisting of $\ku G$.

\begin{theorem}\label{prop:c5c20-pointed}	
The Hopf algebras from Table \ref{tab:gp-th-data-cccv}  have two dual non-zero Yetter-Drinfeld module $V$ with $\dim \toba(V)  = 1280$. 
By bosonization, we get new  Hopf algebras with the dual Chevalley property of dimension 128000. 
\qed
\end{theorem}

The liftings of $\toba(V) \# \ku G$, where $V$ is as above, are classified in \cite[Theorem 6.4, Theorem 6.5]{GVay}. Indeed, let $(x_j)_{0\leqslant j \leqslant 4}$ be the basis of $V$ as in \emph{loc. cit.} Let
\begin{align*}
\Ss= \{(\lambda_1, \lambda_2, \lambda_3) \in \ku^2 \text{ satisfying  \cite[(29), (33)]{GVay}}\}.
\end{align*}
For $(\lambda_1, \lambda_2, \lambda_3) \in \Ss$,  let $H(\lambda_1,\lambda_2, \lambda_3)$ be 
$T(V)\ \# \ku G$ modulo the ideal generated by
\begin{align*}
 x^2_0 &- \lambda_1 (1 - g^2_0), &  & x_0x_1 + x_2x_0 + x_3x_2 +x_1x_3- \lambda_2(1 - g_0 g_1) & & \text{and}
\end{align*}
$$x_1x_0x_1x_0 + x_0x_1x_0x_1 - s_X - \lambda_3(1 - g_0^2 g_1 g_2), \,\, \text{for}$$ 
$$s_X = \lambda_2(x_1x_0 + x_0x_1) + \lambda_1 g_1^2(x_3x_0 + x_2x_3) - \lambda_1 g_0^2(x_2x_4 + x_1x_2) + \lambda_2\lambda_1 g_0^2(1 - g_1 g_2);$$ or by 
\begin{align*}
    x^2_0 & -\lambda_1(1 - g_0^2), &  & x_1x_0 +x_0x_2 +x_2x_3 +x_3x_1 -\lambda_2(1 - g_0g_1) & & \text{and}
\end{align*}
$$x_0x_2x_3x_1 + x_1x_4x_3x_0 - s^\prime_X - \lambda_3 (1 - g_0^2 g_1 g_3), \,\, \text{for}$$
 $$s^\prime_X = \lambda_2(x_0x_1 + x_1x_0) - \lambda_1 g_1^2(x_3x_2 + x_0x_3) - \lambda_1 g_0^2(x_3x_4 + x_1x_3) + \lambda_1\lambda_2(g_1^2 + g_0^2 - 2g_0^2 g_1 g_3).$$

Then
\begin{itemize}[leftmargin=*]
\item $H(\lambda_1,\lambda_2, \lambda_3)$ is a lifting of 
$\toba(V) \# \ku G$,

\item any lifting of $\toba(V)\ \# \ku G$ is $\simeq$ to $H(\lambda_1,\lambda_2, \lambda_3)$ for some $(\lambda_1, \lambda_2, \lambda_3) \in \Ss$,

\item $H(\lambda_1,\lambda_2, \lambda_3) \simeq H(\lambda'_1,\lambda'_2, \lambda'_3)$ iff there exists $\mu \in \ku^{\times}$ such that $(\lambda_1,\lambda_2, \lambda_3) = \mu (\lambda'_1,\lambda'_2, \lambda'_3)$.
\end{itemize}

\begin{obs}\label{obs-lifting-c5c20} Let $H$ be the Hopf algebra corresponding to \ref{sec:c5c20}.\ref{cccv-Twist}. The classification of all liftings of $\toba(V) \# H$ follows from \cite[Th. 6.4, Th. 6.5]{GVay}. Namely, let $J \in \ku G \otimes \ku G$ such that $H\simeq (\ku G)^J$. Then
\begin{itemize}[leftmargin=*]
\item $H(\lambda_1,\lambda_2, \lambda_3)^J$ is a lifting of 
$\toba(V)\ \# H$, for every $(\lambda_1, \lambda_2, \lambda_3) \in \Ss$,

\item any lifting of $\toba(V)\ \# H$ is $\simeq$ to $H(\lambda_1,\lambda_2, \lambda_3)^J$ for some $(\lambda_1, \lambda_2, \lambda_3) \in \Ss$,

\item $H(\lambda_1,\lambda_2, \lambda_3)^J \simeq H(\lambda'_1,\lambda'_2, \lambda'_3)^J$ iff there is $\mu \in \ku^{\times}$ such that $(\lambda_1,\lambda_2, \lambda_3) = \mu (\lambda'_1,\lambda'_2, \lambda'_3)$.
\end{itemize}
\end{obs}

\section{Group-theoretical Hopf algebras over $\mathbb{S}_4$}\label{sec:S4}
The classification of the \fd{} pointed Hopf algebras over $\sk$ was completed in \cite{GG}; there are exactly three non-zero Yetter-Drinfeld modules over $\ku\sk$ whose Nichols algebra is finite-dimensional  and all admit non-trivial deformations. The underlying rack and cocycle are $(\Oc^4_2,   -1)$, $(\Oc^4_2, \chi)$ or $(\Oc^4_4, -1)$.
Here we deal with the group-theoretical Hopf algebras over $G= \mathbb{S}_4$;  since   $\Out \mathbb{S}_4 = 1$,  we need to describe  all group-theoretical data for $G$ up to conjugacy,
cf.  Lemma \ref{prop:algebras gt equivalentes}.

\begin{prop}\label{prop:S4} The classification of the non-trivial group-theoretical Hopf algebras over $\mathbb{S}_4$ is given by the group-theoretical data in Table \ref{tab:gp-th-data-sk}.
\end{prop}

\newcounter{itmsk}
\renewcommand{\theitmsk}{\alph{itmsk}}
\newcommand{\itemsk}[1]{\refstepcounter{itmsk}\ref{sec:S4}.\theitmsk\label{#1}}

\begin{table}[h]
	\begin{center}
		\caption{Group-theoretical data for $\sk$}\label{tab:gp-th-data-sk}
		
		\begin{tabular}{|c|c|c|c|c|c|c|}
			\hline     $\#$ & $F$ &  $\alpha$ & $\Gamma$ & $\beta$ &  $G(H)$
			\\ \hline
\itemsk{sk-Uno} & $ \langle (34), (13)(24)\rangle \simeq  D_4$ & 1 & $\langle (243) \rangle  \simeq C_3$ & 1 &   $C_2 \times C_2$
\\ \hline
\itemsk{sk-Dos} & $\langle (243) \rangle\simeq C_3$ & 1 & $ \langle (34), (13)(24)\rangle \simeq  D_4 $ & 1 &  $\st$
\\ \hline
\itemsk{sk-Tres} & $\langle (12),(34)\rangle\simeq C_2\times C_2$ & $\neq 1$ & $\sk$ & $1$ & $D_4$ 
\\ \hline
\itemsk{sk-Cuatro} & $\sk$ & 1  & $\langle (12),(34)\rangle$ & $\neq 1$  &   $C_2$
\\ \hline
		\end{tabular}
		
	\end{center}
\end{table}

\begin{proof} Since  $|\mathbb{S}_4|=2^3\times 3$, every non-trivial subgroup that admits a non-degenerated 2-cocycle is isomorphic to $C_2\times C_2$. 
Let $(F,\alpha, \Gamma, \beta)$ be a group-theoretical data for $\sk$; then 
$F \cap \Gamma$ is either trivial or else $\simeq C_2\times C_2$.

\noindent\emph{Case 1. } Assume that $F \cap \Gamma = 1$, i.e. $(F, \Gamma)$
is an exact exact factorization. 
Up to conjugation, either $(F, \Gamma)$ or $(\Gamma, F)$ is one of
\begin{align}
\label{sk-uno}
(\langle (34) \rangle, \langle  (13)(24), (243) \rangle) & \simeq 
(C_2, \ac)
\\ \label{sk-tres}
(\langle (243) \rangle, \langle (34), (13)(24)\rangle) & \simeq
(C_3, D_4),
\\ \label{sk-cinco}
(\langle (1324) \rangle, \langle (34), (243)\rangle) & \simeq
(C_4, \st),
\\ \label{sk-cien}
(\langle (14)(23), (13)(24) \rangle, \langle  (34), (243)  \rangle) &\simeq (C_2\times C_2, \mathbb{S}_3).
\end{align}
 If $(F, \Gamma)$ is as in \eqref{sk-cien}, $\langle (14)(23), (13)(24) \rangle$ $\lhd$ $\mathbb{S}_4$ hence $\gth{\mathbb{S}_4}{F}{\Gamma}{\alpha}{\beta}
$  is cocommutative. If $(F, \Gamma)$ is as in the remaining cases,  $F$ is not normal in $\mathbb{S}_4$, $\mathbb{A}_4$ is non-abelian and the others $\Gamma$ are not normal, hence $\gth{\mathbb{S}_4}{F}{\Gamma}{\alpha}{\beta}$  is non-trivial, cf. Lemma \ref{prop:algebras gt equivalentes} \ref{item:cocommutative}. 
Also $\gth{\mathbb{S}_4}{F}{\Gamma}{\alpha}{\beta}\simeq \gth{\mathbb{S}_4}{F}{\Gamma}{1}{1}$ by Remark \ref{obs:aux-opext}.

\noindent\emph{Case 2. } Assume that $F \cap \Gamma \simeq C_2\times C_2$. Up to conjugation, either $(F, \Gamma)$ or $(\Gamma, F)$ is one of
\begin{align}
\label{sk-siete}
(\langle (12), (34) \rangle, G ) & \simeq (C_2 \times C_2, \sk)
\\\label{sk-sietebis}
(\langle (14)(23), (13)(24) \rangle, G ) & \simeq (C_2 \times C_2, \sk)
\\ \label{sk-nueve}
(\langle (14)(23), (34) \rangle, \langle (13)(24), (243)\rangle) & \simeq
(D_4, \ac).
\end{align}

If $(F, \Gamma)$ is as in \eqref{sk-sietebis}, $H = \gth{\mathbb{S}_4}{F}{\mathbb{S}_4}{\alpha}{1}$  is cocommutative by Lemma \ref{prop:algebras gt equivalentes} \ref{item:cocommutative};
but if it is as in \eqref{sk-siete}, then $H$ (a twist of $\ku \sk$) is non-trivial  since $F \ntriangleleft \sk$. 

We next deal with  $(F, \Gamma)$ as in \eqref{sk-nueve}. If $\alpha\in H^2(D_4,\ku^\times)\simeq C_2 \simeq H^2(\mathbb{A}_4,\ku^\times) \ni \beta$, then $\alpha_{\vert F\cap\Gamma}\cdot {\beta_{\vert F\cap\Gamma}}^{-1} \neq 1$ iff either $\alpha \neq 1$ and $\beta =1$, or vice versa. By Lemma \ref{prop:algebras gt equivalentes} \ref{item:cocommutative}, $\gth{\mathbb{S}_4}{F}{\Gamma}{\alpha}{1}$  and  $\gth{\mathbb{S}_4}{F}{\Gamma}{1}{\beta}$ are non-trivial, since $F$ and $\mathbb{A}_4$ are non-abelian. By Remark \ref{obs:aux-opext}, $\gth{\mathbb{S}_4}{F}{\Gamma}{\alpha}{1}\simeq \gth{\mathbb{S}_4}{F}{\Gamma}{1}{\beta}$. 

If $(\Gamma, F)$ is as in any of the cases \eqref{sk-uno}, \dots, \eqref{sk-nueve} above, then $\gth{\mathbb{S}_4}{\Gamma}{F}{\beta}{\alpha}$ is dual to $H = \gth{\mathbb{S}_4}{F}{\Gamma}{\alpha}{\beta}$  by Lemma \ref{prop:algebras gt equivalentes} \ref{item:duals}. In conclusion, we have the following non-trivial Hopf algebras (with a slight abuse of notation):
\begin{align*}
\eqref{sk-uno}:\, &\gth{\mathbb{S}_4}{C_2}{\ac}{1}{1}, & \eqref{sk-uno}^*:\, &\gth{\mathbb{S}_4}{\ac}{C_2}{1}{1};
\\ \eqref{sk-tres}:\,  &\gth{\mathbb{S}_4}{C_3}{D_4}{1}{1}, & \eqref{sk-tres}^*:\, &\gth{\mathbb{S}_4}{D_4}{C_3}{1}{1};
\\
\eqref{sk-cinco}:\,  &\gth{\mathbb{S}_4}{C_4}{\st}{1}{1}, & \eqref{sk-cinco}^*:\, &\gth{\mathbb{S}_4}{\st}{C_4}{1}{1};
\\
\eqref{sk-siete}:\,  &\gth{\mathbb{S}_4}{C_2 \times C_2}{\sk}{\alpha}{1}, & \eqref{sk-siete}^*:\, &\gth{\mathbb{S}_4}{\sk}{C_2 \times C_2}{1}{\beta};
\\
\eqref{sk-nueve}:\,  &\gth{\mathbb{S}_4}{D_4}{\ac}{1}{\beta}, & \eqref{sk-nueve}^*:\, &\gth{\mathbb{S}_4}{\ac}{D_4}{\alpha}{1}.
\end{align*}

 By Lemma \ref{Lema twist equivalencia} below, the Hopf algebras \eqref{sk-uno}, \eqref{sk-tres}$^*$ and \eqref{sk-nueve}  are twist-equivalent, but since $\C(\mathbb{S}_4,1;\langle (234)\rangle,1)$ admits a unique fiber functor, then all of them are isomorphic--this gives \ref{sec:S4}.\ref{sk-Uno}, with dual \ref{sec:S4}.\ref{sk-Dos}. Similarly, by Lemma \ref{Lema twist equivalencia}, the Hopf algebras \eqref{sk-cinco} and \eqref{sk-siete} are  twist-equivalent, hence isomorphic because $\ku \sk$ admits a unique non-trivial twisting--this gives \ref{sec:S4}.\ref{sk-Tres}, with dual \ref{sec:S4}.\ref{sk-Cuatro}.
The computation of the various $G(H)$ is performed via Lemma \ref{lema:aux-G(H)}; hence the Hopf algebras in Table \ref{tab:gp-th-data-sk} are not isomorphic to each other.
\end{proof}

\begin{lema}\label{Lema twist equivalencia}
\begin{enumerate}[leftmargin=*,label=\rm{(\roman*)}] 
\item\label{item:twist-equiv-i}  $\C(\mathbb{S}_4,1;\langle (34), (243) \rangle,1)\cong_\otimes \C(\mathbb{S}_4,1;\mathbb{S}_4,\alpha)\cong_\otimes \Rep{\mathbb{S}_4}$, where $\alpha \in H^2(\sk,\ku^\times)$. 

\item\label{item:twist-equiv-ii}  $\C(\mathbb{S}_4,1;\langle (234)\rangle,1)\cong_\otimes \C(\mathbb{S}_4,1;\mathbb{A}_4,\beta)$, where $\beta \in H^2(\mathbb{A}_4,\ku^\times)$.
\end{enumerate}

\end{lema}
\begin{proof}
Let us recall some results related with invertible bimodules over pointed fusion categories and the tensor product of their module categories:

If $\mathcal{X}$ is an invertible $\vect_G$-bimodule category then as right $\vect_G$-module category $\mathcal{X}\cong \M(A,\alpha)$, where $A\lhd G$ is abelian and $\alpha \in H^2(A,\ku^\times)^{\operatorname{ad} G}$, cf \cite[Corollary 7.11]{GJ}. In the case of $G=\sk$, there  are invertible bimodule categories $\mathcal{X}$ such that as right $\vect_{\mathbb{S}_4}$-module category $\mathcal{X}=\M(N,\alpha)$, where $N$ is the Klein normal subgroup of $\sk$, cf \cite[Subsection 8.2]{NB}. 

The rank of $\M(F,\alpha)\boxtimes_{\vect_{\mathbb{S}_4}}\M(\Gamma,\beta)$ can be calculated as 
follows: Let $X:=F\backslash G$ and $Y:= G/\Gamma$ the right and left transitive $G$-sets associated. The groups $G$ acts on $X\times Y$ as $g \cdot (x,y)= (xg^{-1},gy)$. Let $\{(x_i,y_i)\}_{i\in F\backslash G/\Gamma}$ be a set of representatives of the orbits of $G$ in $X \times  Y$ (the set of $G$-orbits  is in correspondence with the $(F, \Gamma)$-double cosets), then there is a bijective correspondence between simple objects in $\M(F,\alpha)\boxtimes_{\vect_{\mathbb{S}_4}}\M(\Gamma,\beta)$ and irreducible representations  of $\ku_{\alpha_i}\operatorname{Stab}_G (x_i,y_i)$, where $\alpha_i$ are certain 2-cocycles associated with $\alpha$ and $\beta$, cf \cite[Theorem 7.15, Corollary 7.16]{GJ}.

\ref{item:twist-equiv-i} By Lemma \ref{lema:aux-twist} we have to see that there is an invertible  vec$_{\sk}$-bimodule category $\mathcal{X}$ such that
$$\mathcal{X}\boxtimes_{\vect_{\sk}}\M(\langle (34), (243) \rangle,1)\cong \M(\mathbb{S}_4,\alpha),$$  since $\C(\mathbb{S}_4,1;\mathbb{S}_4,\alpha)\cong_\otimes \Rep{\mathbb{S}_4}$ for all $\alpha \in H^2(\sk,\ku^\times)$. 

Let  $\mathcal{X}$ be an invertible bimodule categories such that as right $\vect_{\mathbb{S}_4}$-module  category $\mathcal{X}=\M(N,\alpha)$, where $N$ is the Klein normal subgroup of $\sk$. The rank of $\M(N,\alpha)\boxtimes_{\vect_{\mathbb{S}_4}}\M(\langle (34), (243) \rangle,1)$ is one, since $N$ and $\langle (34), (243) \rangle$ are an exact factorization of $\mathbb{S}_4$.

The module categories $\M(\mathbb{S}_4,\alpha)$ are characterized as the  $\vect_{\mathbb{S}_4}$-module categories of rank one. So, $\mathcal{X}\boxtimes_{\vect_{\sk}}\M(\langle (34), (243) \rangle,1)\cong \M(\mathbb{S}_4,\alpha)$.

\ref{item:twist-equiv-ii} Again, let $\mathcal{X}$ be an invertible bimodule categories such that as right $\vect_{\mathbb{S}_4}$-module  category $\mathcal{X}=\M(N,\alpha)$, where $N$ is the Klein normal subgroup of $\sk$. Then rank $\M(N,\alpha)\boxtimes_{\vect_{\mathbb{S}_4}}\M(\langle \langle (234)\rangle) \rangle,1)$ is two, since there are only two $(\langle (234)\rangle, N)$-double cosets and their stabilizers are trivial.

The module categories $\M(\mathbb{A}_4,\beta)$ are characterized as the  $\vect_{\mathbb{S}_4}$-module categories of rank two. So, $\mathcal{X}\boxtimes_{\vect_{\sk}}\M(\langle (234)\rangle,\alpha)\cong \M(\mathbb{A}_4,\beta)$.
\end{proof}

Then the Hopf algebra associated to \ref{sec:S4}.\ref{sk-Tres} is a twisting of $\ku \mathbb{S}_4$.

\begin{theorem}\label{prop:S4-pointed}	
The group-theoretical Hopf algebras in Table \ref{tab:gp-th-data-sk} have exactly 3 non-zero Yetter-Drinfeld modules  whose Nichols algebra is finite-dimensional. By bosonization, we get new  Hopf algebras with the dual Chevalley property of dimension  13824.\qed
\end{theorem}

The liftings of $\toba(V) \# \ku \mathbb{S}_4$, where $V$ is as above, are classified in \cite[Proposition 5.3]{GG}. Indeed, for the ql-data \cite[Def. 3.5]{GG} 
\begin{itemize}
\item $\mathcal{Q}_4^{-1}[t]=(\sk , \mathcal{O}^4_2,−1,\cdot,\iota, \{0, \Lambda,\Gamma\})$,
\item $\mathcal{Q}_4^\chi[\lambda]= (\mathbb{S}_4, \mathcal{O}_2^4, \chi,\cdot, \iota, \{ 0, 0, \lambda\})$ and 
\item $\mathcal{D}[t]= ((\sk,\mathcal{O}_4^4, −1, \cdot, \iota, \{\Lambda, 0, \Gamma\}))$,
\end{itemize} where $\Lambda, \Gamma, \lambda \in \ku$, $t = (\Lambda, \Gamma)$, let $\mathcal{H}(\mathcal{Q}_4^{-1}[t])$, $\mathcal{H}(\mathcal{Q}_4^\chi[\lambda])$ and $\mathcal{H}(\mathcal{D}[t])$, respectively, be the algebras presented by generators and relations
\begin{itemize}
\item $\{a_i, H_r : i\in \mathcal{O}_4^2, r\in \mathbb{S}_4\}$,
\begin{align*}
    & H_e=1,   \,\,\, \,\,\,        H_rH_s=H_{rs},     \,\,\, \,\,\,     r, s\in \mathbb{S}_4,\\
    & H_j a_i = -a_{jij} H_j,  \,\,\, \,\,\, i, j\in \mathcal{O}^4_2,  \\
    & a^2_{(12)}=0, \\
    & a_{(12)}a_{(34)} + a_{(34)}a_{(12)}  = \Gamma(1 - H_{(12)}H_{(34)}), \\
    & a_{(12)}a_{(23)} + a_{(23)}a_{(13)} + a_{(13)}a_{(12)} = \Lambda(1 - H_{(12)}H_{(23)}), 
    \end{align*} 
\item $\{a_i, H_r : i \in \mathcal{O}^4_2, r \in \mathbb{S}_4\}$,
\begin{align*}
    & H_e = 1, \,\,\, \,\,\,   H_rH_s = H_{rs}, \,\,\, \,\,\,   r, s \in \sk,\\
    & H_j a_i = \chi_i(j)a_{jij}H_j, \,\,\, \,\,\,   i, j \in \mathcal{O}^4_2,\\
    & a^2_{(1 2)} = 0,\\
    & a_{(12)}a_{(34)} - a_{(34)}a_{(12)} = 0,\\
    & a_{(12)}a_{(23)} - a_{(23)}a_{(13)} - a_{(13)}a_{(12)} = \lambda (1 - H_{(12)}H_{(23)});
\end{align*} 
\item $\{a_i, H_r : i \in \mathcal{O}_4^4, r\in \sk\}$,
\begin{align*}
& H_e=1, \,\,\, \,\,\, H_r H_s = H_{rs}, \,\,\, \,\,\, r,s \in \sk,\\
& H_j a_i = -a_{jij} H_j, \,\,\, \,\,\, i \in \mathcal{O}_4^4, j \in \mathcal{O}_2^4,\\
&a^2_{(1234)} = \Gamma (1 - H_{(13)}H_{(24)}),\\
& a_{(1234)} a_{(1432)} + a_{(1432)}a_{(1234)} = 0, \\
& a_{(1234)}a_{(1243)} + a_{(1243)}a_{(1423)} + a_{(1423)}a_{(1234)} = \Gamma(1 - H_{(12)}H_{(13)}).
\end{align*}
\end{itemize}
Then 
\begin{itemize}
\item these algebras are liftings of $\toba(V) \# \ku \mathbb{S}_4$,
\item any lifting is isomorphic to one of these algebras,
\item $\mathcal{H}(\mathcal{Q}_4^{-1}[t])\simeq \mathcal{H}(\mathcal{Q}_4^{-1}[t^\prime])$ iff $t\neq 0$ and $t = t^\prime \in \mathbb{P}^1_\ku$ or if $t = t^\prime = (0,0)$, and the same holds for $\mathcal{H}(\mathcal{D}[t])$; $\mathcal{H}(\mathcal{Q}_4^\chi[\lambda])\simeq \mathcal{H}(\mathcal{Q}_4^\chi[1])$, $\forall\lambda \in \ku^\times$ and $\mathcal{H}(\mathcal{Q}_4^\chi[1])\not\simeq \mathcal{H}(\mathcal{Q}_4^\chi[0])$, \cite[Lemma 6.1]{GG}.
\end{itemize}

\begin{obs}\label{obs-lifting-s4} Let $H$ be the Hopf algebra corresponding to \ref{sec:S4}.\ref{sk-Tres}. The classification of all liftings of $\toba(V) \# H$ follows from \cite[Prop. 5.3]{GG}. Namely, let $J \in \ku \sk \otimes \ku \sk$ such that $H\simeq (\ku \sk)^J$. Then
\begin{itemize}[leftmargin=*]
\item $\mathcal{H}(\mathcal{Q}_4^{-1}[t])^J$, $\mathcal{H}(\mathcal{Q}_4^\chi[\lambda])^J$ and $\mathcal{H}(\mathcal{D}[t])^J$ are liftings of 
$\toba(V)\ \# H$,

\item any lifting of $\toba(V)\ \# H$ is $\simeq$ to $\mathcal{H}(\mathcal{Q}_4^{-1}[t])^J$, $\mathcal{H}(\mathcal{Q}_4^\chi[\lambda])^J$ or $\mathcal{H}(\mathcal{D}[t])^J$,

\item $\mathcal{H}(\mathcal{Q}_4^{-1}[t])^J\simeq \mathcal{H}(\mathcal{Q}_4^{-1}[t^\prime])^J$ iff $t\neq 0$ and $t = t^\prime \in \mathbb{P}^1_\ku$ or if $t = t^\prime = (0,0)$, and the same holds for $\mathcal{H}(\mathcal{D}[t])^J$; $\mathcal{H}(\mathcal{Q}_4^\chi[\lambda])^J\simeq \mathcal{H}(\mathcal{Q}_4^\chi[1])^J$, $\forall\lambda \in \ku^\times$ and $\mathcal{H}(\mathcal{Q}_4^\chi[1])^J\not\simeq \mathcal{H}(\mathcal{Q}_4^\chi[0])^J$.
\end{itemize}
\end{obs}

\section{Group-theoretical Hopf algebras over  $\mathbb{S}_5$}\label{sec:S5} 

The classification of the \fd{} Nichols algebras over $\sco$ is unknown; there are two non-zero Yetter-Drinfeld modules over $\ku\sco$ with finite-dimensional Nichols algebra  \cite{FK, G-zoo, GG} and one open case \cite{AFGV-ampa}. The underlying rack and cocycles are $(\Oc^5_2$, $-1$) or ($\Oc^5_2$, $\chi$).
Here we deal with the group-theoretical Hopf algebras over $G= \mathbb{S}_5$  (up to conjugacy because $\Out G = 1$, cf.  Lemma \ref{prop:algebras gt equivalentes}).

\begin{prop}\label{prop:S5} The classification of the non-trivial group-theoretical Hopf algebras over $\mathbb{S}_5$ is given by the group-theoretical data in Table \ref{tab:gp-th-data-s5}.
\end{prop}

\newcounter{itmsco}
\renewcommand{\theitmsco}{\alph{itmsco}}
\newcommand{\itemsco}[1]{\refstepcounter{itmsco}\ref{sec:S5}.\theitmsco\label{#1}}

\begin{table}[h]
	\begin{center}
		\caption{Group-theoretical data for $\sco$}\label{tab:gp-th-data-s5}
		
		\begin{tabular}{|c|p{2.9cm}|c|c|c|c|c|}
			\hline     $\#$ & $F$ &  $\alpha$ & $\Gamma$ & $\beta$ &  $G(H)$
			\\ \hline
\itemsco{s5-uno}  & $\langle (45) \rangle \simeq C_2$ & 1 & $\langle (12345), (345) \rangle \simeq  \mathbb{A}_5$ & 1  & $C_2\times \mathbb{S}_3$
\\ \hline
\itemsco{s5-dos} & $\langle (12345), (345) \rangle$ & 1 & $\langle (45) \rangle $ & 1 & $ C_2 $
\\ \hline
\itemsco{s5-tres} &  $\langle (12345)\rangle \simeq C_5$ & 1 &  $\langle (345), (2435)\rangle \simeq \mathbb{S}_4$ & 1 & $C_5\rtimes C_4$
\\ \hline
\itemsco{s5-cuatro} & $\langle (345), (2435)\rangle$ & 1 & $\langle (12345)\rangle$ & 1 &  $C_2$
\\ \hline
\itemsco{s5-cinco} & {\small $\langle (345), (45) \rangle \simeq \mathbb{S}_3$} & 1 & 
{\small $\langle (12345), (2354)\rangle \simeq C_5\rtimes C_4$} &  1 & $C_2\times C_2$
\\ \hline
\itemsco{s5-seis} & $\langle (12345), (2354)\rangle $ & 1 & $\langle (345), (45) \rangle$ & 1  & $C_4 $
\\ \hline
\itemsco{s5-siete} & {\small $\langle (12)(345)\rangle\simeq C_6$} & 1 & {\small $\langle (12345), (2354)\rangle \simeq C_5\rtimes C_4$} & 1  & $D_6$
\\ \hline
\itemsco{s5-ocho} & $\langle (12345), (2354)\rangle$  &  1 & $\langle (12)(345)\rangle$ &  1  & $C_4 $
\\ \hline
\itemsco{s5-nueve} & {\small $\langle (23)(45), (24)(35)\rangle$} \newline {\small $\simeq C_2\times C_2$} & $\neq 1$ & $\mathbb{S}_5$ & 1  &  {\small $(C_2\times C_2)\rtimes_\nu\mathbb{S}_3$} 
\\ \hline
\itemsco{s5-diez}  & $\mathbb{S}_5$ & 1 & $\langle (23)(45), (24)(35)\rangle$ & $\neq 1$  & $C_2$
\\ \hline
\itemsco{s5-once} & {\small $\langle (45), (23) \rangle$\newline $\simeq C_2\times C_2$} & $\neq 1$ &  $\mathbb{S}_5$ & 1  & $D_4$ 
\\ \hline
\itemsco{s5-doce}  & $\mathbb{S}_5$  & 1  &  $\langle (45), (23) \rangle$ & $\neq 1$   & $C_2$
\\ \hline
\itemsco{s5-trece}  & $\langle (45), (24)(35)\rangle$ \newline $\simeq D_4$ & $\neq 1$  &  $\mathbb{A}_5$ &   1 & $C_2\times C_2$ 
\\ \hline
\itemsco{s5-quince}  & $\mathbb{A}_5$  & 1 & $\langle (45), (24)(35)\rangle$ & $\neq 1$  & $ C_2$ 
\\ \hline
		\end{tabular}
		
	\end{center}
\end{table}

\pf Since  $|\mathbb{S}_5|=2^3\times 3\times 5$, every non-trivial subgroup that admits a non-degenerated 2-cocycle is isomorphic to $C_2\times C_2$. Let $(F,\alpha, \Gamma, \beta)$ be a group-theoretical data for $\sco$; then 
$F \cap \Gamma$ is either trivial or else $\simeq C_2\times C_2$.

 \noindent\emph{Case 1. } Assume that $F \cap \Gamma = 1$, i.e. $(F, \Gamma)$
is an exact exact factorization. Up to conjugation,  $(F,\Gamma)$  is either of \ref{sec:S5}.\ref{s5-uno}, \ref{sec:S5}.\ref{s5-tres}, \ref{sec:S5}.\ref{s5-cinco} or \ref{sec:S5}.\ref{s5-siete}, or their transposes \ref{sec:S5}.\ref{s5-dos}, \ref{sec:S5}.\ref{s5-cuatro}, \ref{sec:S5}.\ref{s5-seis} or \ref{sec:S5}.\ref{s5-ocho}. Since $\mathbb{A}_5$ is non-abelian and the others subgroups in this list are not normal, then $\gth{\mathbb{S}_5}{F}{\Gamma}{\alpha}{\beta}$  is non-trivial by Lemma \ref{prop:algebras gt equivalentes} \ref{item:cocommutative}. Also $\gth{\mathbb{S}_5}{F}{\Gamma}{\alpha}{\beta}\simeq \gth{\mathbb{S}_5}{F}{\Gamma}{1}{1}$ by Remark \ref{obs:aux-opext}.

\noindent\emph{Case 2. } Assume that $F \cap \Gamma \simeq C_2\times C_2$. Up to conjugation, either $(F, \Gamma)$ or $(\Gamma, F)$ is \ref{sec:S5}.\ref{s5-nueve}, \ref{sec:S5}.\ref{s5-once} or \ref{sec:S5}.\ref{s5-trece}. By Lemma \ref{prop:algebras gt equivalentes} \ref{item:cocommutative}, since in the first two cases $F\ntriangleleft \mathbb{S}_5$, then $\gth{\mathbb{S}_5}{F}{\mathbb{S}_5}{\alpha}{1}$ is non-trivial. We next deal with $(F, \Gamma)$ as in \ref{sec:S5}.\ref{s5-trece}. If $\alpha\in H^2(D_4,\ku^\times)\simeq C_2 \simeq H^2(\mathbb{A}_5,\ku^\times)  \ni \beta$, then $\alpha_{\vert F\cap\Gamma}\cdot {\beta_{\vert F\cap\Gamma}}^{-1} \neq 1$ iff either $\alpha \neq 1$ and $\beta =1$, or vice versa. By Lemma \ref{prop:algebras gt equivalentes} \ref{item:cocommutative}, $\gth{\mathbb{S}_5}{F}{\mathbb{A}_5}{\alpha}{1}$  and $\gth{\mathbb{S}_5}{F}{\mathbb{A}_5}{1}{\beta}$  are non-trivial, since $F$ and $\mathbb{A}_5$ are non-abelian. By Remak \ref{obs:aux-opext}, $\gth{\mathbb{S}_5}{F}{\mathbb{A}_5}{\alpha}{1} \simeq \gth{\mathbb{S}_5}{F}{\mathbb{A}_5}{1}{\beta}$.

If $(\Gamma, F)$ is as in any of the cases \ref{sec:S5}.\ref{s5-uno}, \ref{sec:S5}.\ref{s5-tres}, \ref{sec:S5}.\ref{s5-cinco}, \ref{sec:S5}.\ref{s5-siete}, \ref{sec:S5}.\ref{s5-nueve}, \ref{sec:S5}.\ref{s5-once} and \ref{sec:S5}.\ref{s5-trece}, then $\gth{\mathbb{S}_5}{\Gamma}{F}{\beta}{\alpha}$ is dual to $H = \gth{\mathbb{S}_5}{F}{\Gamma}{\alpha}{\beta}$  by Lemma \ref{prop:algebras gt equivalentes} \ref{item:duals}. The computation of the various $G(H)$ is performed via Lemma \ref{lema:aux-G(H)}; hence the Hopf algebras in Table \ref{tab:gp-th-data-s5} are not isomorphic to each other.
\epf

\begin{theorem}\label{prop:S5-pointed}	
The Hopf algebras from Table \ref{tab:gp-th-data-s5} have two non-zero Yetter-Drinfeld modules $V$ with $\dim \toba(V) < \infty$. 
By bosonization, we get new  Hopf algebras with the dual Chevalley property of dimension  995328000. \qed
\end{theorem}

\section{Group-theoretical Hopf algebras over  $G = \mathbb{A}_4\times C_2$}\label{sec:A4xZ2}  
The classification of the \fd{} Nichols algebras over $G = \mathbb{A}_4\times C_2$ is not known, but there is
$V \in \ydg$ with $\dim \toba(V) = 72$ \cite{G-cm}; the underlying rack is $\T$ and the cocycle is $-1$.  The group $\Aut G$  is isomorphic to $\sk$, via $\varphi: \mathbb{S}_4\to \Aut G$ given by $\varphi(a)(b,i)=(aba^{-1},i)$, for all $a\in \mathbb{S}_4$, $b\in\mathbb{A}_4$, $i\in C_2=\{1, x\}$. Let $M < G$, $$M := \langle ((13)(24),1), ((12)(34), 1),(1,x)\rangle\simeq C_2\times C_2\times C_2.$$

\begin{prop}\label{prop:A4xZ2} The non-trivial group-theoretical Hopf algebras over $G = \mathbb{A}_4\times C_2$ correspond to the group-theoretical data in Table \ref{tab:gp-th-data-a4c2}.
\end{prop}
\begin{question}
Is it true that \ref{sec:A4xZ2}.\ref{A4xZ2-siete} $\simeq$ \ref{sec:A4xZ2}.\ref{A4xZ2-diez}
$\simeq$ \ref{sec:A4xZ2}.\ref{A4xZ2-once} $\simeq$ \ref{sec:A4xZ2}.\ref{A4xZ2-tres} $\simeq$ \ref{sec:A4xZ2}.\ref{A4xZ2-cinco} $\simeq$ \ref{sec:A4xZ2}.\ref{A4xZ2-uno}?
\end{question}

\newcounter{itmalt}
\renewcommand{\theitmalt}{\alph{itmalt}}
\newcommand{\itemalt}[1]{\refstepcounter{itmalt}\ref{sec:A4xZ2}.\theitmalt\label{#1}}

\begin{table}[h]
	\begin{center}
		\caption{Group-theoretical data for $G =\ac \times C_2$.} \label{tab:gp-th-data-a4c2}
		
		\begin{tabular}{|c|p{2cm}|p{2.5cm}|p{2cm}|p{2.5cm}|c|}
			\hline     $\#$ & $F$ &  $\alpha$ & $\Gamma$ & $\beta$ &  $G(H)$ 
			\\ \hline
			\itemalt{A4xZ2-siete}  &  {\small $\langle ((12)(34),x)\rangle$\newline $\simeq C_2$} & 1 &  $\mathbb{A}_4\times 1$ & 1 & {\small $C_2\times C_2\times C_2$}
			\\ \hline
			\itemalt{A4xZ2-ocho}  & $\mathbb{A}_4\times 1$ & 1 & {\small $\langle ((12)(34),x)\rangle$} & 1 & $C_6$
		\\ \hline
			\itemalt{A4xZ2-nueve}  & $\langle ((123), 1)\rangle$\newline $\simeq C_3$ & 1 & $M$ &  {\small $\notin H^2(F,\ku^\times)^{\ad G}$}\newline {\small $\alpha|_{F\cap\Gamma}\neq 1$} & $C_6$
				\\ \hline
			\itemalt{A4xZ2-diez}  & $M$ &  {\small $\notin H^2(F,\ku^\times)^{\ad G}$}\newline {\small $\alpha|_{F\cap\Gamma}\neq 1$} & $\langle ((123), 1)\rangle$ & 1 & {\small $C_2\times C_2 \times C_2$}
			\\ \hline
			\itemalt{A4xZ2-once}  & {\small $\langle ((14)(23), 1),$\newline $((12)(34), x)\rangle$\newline  $\simeq C_2\times C_2$} & 1 & $\langle ( (123), x)\rangle$ \newline$\simeq C_6$  & 1 & {\small $C_2\times C_2 \times C_2$}
			\\ \hline
			\itemalt{A4xZ2-doce}  & $\langle ( (123), x)\rangle$ & 1 & {\small $\langle ((14)(23), 1),$\newline $((12)(34), x)\rangle$} & 1 & $C_6$
				\\ \hline
			\itemalt{A4xZ2-tres}  & {\small $\langle ((13)(24), 1),$\newline $((12)(34), 1) \rangle$\newline $\simeq C_2\times C_2$} & $\neq 1$ & $G$ & 1 & {\small $ C_2\times C_2 \times C_2$}
		\\ \hline
			\itemalt{A4xZ2-cuatro}  & $G$ & 1 & {\small $\langle ((13)(24), 1),$\newline $((12)(34), 1) \rangle$} & $\neq 1$ & $C_6$
\\ \hline	
\itemalt{A4xZ2-cinco}  & {\small $\langle ((13)(24), x),$\newline $((12)(34), 1) \rangle$ \newline $\simeq C_2\times C_2$} & $\neq 1$ & $G$ & 1 & {\small $ C_2\times C_2 \times C_2$}
\\ \hline	
\itemalt{A4xZ2-seis}  & $G$ & 1 & {\small $\langle ((13)(24), x),$\newline $((12)(34), 1) \rangle$} & $\neq 1$ & $C_6$
\\ \hline	
\itemalt{A4xZ2-uno}  & $M$
&  {\small $\notin H^2(F,\ku^\times)^{\ad G}$}\newline {\small $\alpha|_{F\cap\Gamma}\neq 1$} &  $\mathbb{A}_4\times 1$  & 1      & {\small $C_2\times C_2\times C_2$}
\\ \hline
\itemalt{A4xZ2-dos}  & $\mathbb{A}_4\times 1$ & 1  & $M$   &  {\small $\notin H^2(\Gamma,\ku^\times)^{\ad G}$} \newline {\small $\beta|_{F\cap\Gamma}\neq 1$}      &  	{\small	$C_6$}			
			\\ \hline
		\end{tabular}
		
	\end{center}
\end{table}

\pf Since  $|G|=2^3\times 3$, every non-trivial subgroup that admits a non-degenerated 2-cocycle is isomorphic to $C_2\times C_2$. Let $(F,\alpha, \Gamma, \beta)$ be a group-theoretical data for $\sco$; then $F \cap \Gamma$ is either trivial or else $\simeq C_2\times C_2$.
 
 \noindent\emph{Case 1. } Assume that $F \cap \Gamma = 1$, i.e. $(F, \Gamma)$
is an exact exact factorization. Up to automorphism, the exact factorizations  $(F, \Gamma)$ of $\mathbb{A}_4\times C_2$ are either  
\begin{align}
\label{case1-a4c2-uno}( 1\times C_2 ,\mathbb{A}_4\times 1) & \simeq (C_2, \mathbb{A}_4)\\
\label{case1-a4c2-dos}( \langle ((12)(34),x)\rangle, \mathbb{A}_4\times 1) & \simeq (C_2, \mathbb{A}_4),\\ 
\label{case1-a4c2-tres}(\langle ((123), 1)\rangle, M) & \simeq (C_3, C_2\times C_2\times C_2),\\
\label{case1-a4c2-cuatro}(\langle ((14)(23), 1), ((12)(34), 1)\rangle, \langle ( (123), x)\rangle) & \simeq (C_2\times C_2, C_6),\\
\label{case1-a4c2-cinco}(\langle ((14)(23), 1),((12)(34), x)\rangle,\langle ( (123), x)\rangle) & \simeq (C_2\times C_2, C_6), 
\end{align}
or their transposes. If $(F, \Gamma)$ as in \eqref{case1-a4c2-uno} or \eqref{case1-a4c2-cuatro}, then $\gth{G}{F}{\Gamma}{\alpha}{\beta}$ is cocommutative. If $(F, \Gamma)$ is as in \eqref{case1-a4c2-dos}, $F\ntriangleleft G$ and $\Gamma$ is non-abelian, hence $\gth{G}{F}{\Gamma}{\alpha}{\beta}$ is non-trivial. If $(F, \Gamma)$ is as in \eqref{case1-a4c2-cinco}, then $\gth{G}{F}{\Gamma}{\alpha}{\beta}$ is non-trivial, since $F, \Gamma \ntriangleleft G$. We next deal with $(F, \Gamma)$ as in \eqref{case1-a4c2-tres}. There are six elements in $H^2(M,\Bbbk^\times)-H^2(M,\Bbbk^\times)^{\mathrm{ad} G}$ and $\langle \mathrm{ad}_a\rangle < \Aut G$ where $a\in G-M$, acts transitively on $H^2(M,\Bbbk^\times)-H^2(M,\Bbbk^\times)^{\mathrm{ad} G}$, so we get isomorphic Hopf algebras. For such an $\beta$, $\gth{G}{F}{\Gamma}{1}{\beta}$ is non-trivial, as $F \ntriangleleft G$.

\noindent\emph{Case 2. }$F \cap \Gamma \simeq C_2\times C_2$. Up to automorphism, either $(F, \Gamma)$ or $(\Gamma, F)$ is one of

\begin{align}
\label{case2-a4c2-uno}
(\langle ((13)(24), 1), ((12)(34), 1) \rangle, G) & \simeq (C_2\times C_2, G), \\
\label{case2-a4c2-dos}
(\langle (1 ,x), ((12)(34), 1) \rangle, G) & \simeq (C_2\times C_2, G), \\
\label{case2-a4c2-tres}
(\langle ((13)(24), x), ((12)(34), 1) \rangle, G) & \simeq (C_2\times C_2, G),  \\
\label{case2-a4c2-cuatro}
(M,  \mathbb{A}_4\times 1)  &\simeq (C_2\times C_2 \times C_2, \mathbb{A}_4).
\end{align} 

If $(F, \Gamma)$ is as in \eqref{case2-a4c2-uno} then $\gth{G}{F}{G}{\alpha}{1}$ is cocommutative. If $(F, \Gamma)$ is as in \eqref{case2-a4c2-dos} or \eqref{case2-a4c2-tres}, $F\ntriangleleft G$, hence $\gth{G}{F}{G}{\alpha}{1}$ is non-trivial.  We next deal with  $(F, \Gamma)$ as in \eqref{case2-a4c2-cuatro}. For $\alpha \in X$, $\gth{G}{M}{\mathbb{A}_4\times 1}{\alpha}{1}$ is not cocommutative; and is not commutative, since $\mathbb{A}_4\times 1$ is non-abelian. $\gth{G}{M}{\mathbb{A}_4\times 1}{1}{\beta}$ is cocommutative.
In conclusion, we have the non-trivial Hopf algebras described by the group-theoretical data in Table \ref{tab:gp-th-data-a4c2}.
\epf


Here  \ref{sec:A4xZ2}.\ref{A4xZ2-tres} and \ref{sec:A4xZ2}.\ref{A4xZ2-cinco} are twistings of $\ku G$. 

\begin{theorem}\label{prop:A4xZ2-pointed}	
	The Hopf algebras from Table \ref{tab:gp-th-data-a4c2} have a non-zero Yetter-Drinfeld module $V$ with $\dim \toba(V) = 72$. 
	By bosonization, we get new  Hopf algebras with the dual Chevalley property of dimension 1728.\qed
\end{theorem}

The liftings of $\toba(V) \# \ku G$, where $V$ is as above, are classified in \cite[Theorem 6.3]{GVay}. Indeed, let $(x_j)_{0\leqslant j \leqslant 2}$ be the basis of $V$ as in \emph{loc. cit.} Let
\begin{align*}
\Ss= \{(\lambda_1, \lambda_2, \lambda_3) \in \ku^2 \text{ satisfying  \cite[(29), (33)]{GVay}}\}.
\end{align*}
For $(\lambda_1, \lambda_2, \lambda_3) \in \Ss$,  let $H(\lambda_1,\lambda_2, \lambda_3)$ be 
$T(V)\ \# \ku G$ modulo the ideal generated by
\begin{align*}
 x^2_0 &- \lambda_1 (1 - g^2_0), &  & x_0x_1 + x_1x_2 + x_2x_0 - \lambda_2(1 - g_0 g_1) & & \text{and}
\end{align*}
$$x_2x_1x_0x_2x_1x_0 + x_1x_0x_2x_1x_0x_2+x_0x_2x_1x_0x_2x_1 - s_X - \lambda_3(1 - g_0^3 g_1^3),$$ 
where
\begin{align*}s_X  = & \lambda_2(x_2x_1x_0x_2 + x_1x_0x_2x_1 + x_0x_2x_1x_0) - \lambda_2^3(g_0 g_1 - g_0^3 g_1^3) \\
 +& \lambda^2_1 g_0^2(g_3^2(x_2x_3 + x_0x_2) + g_1 g_3(x_2x_1 + x_1x_3) + g_1^2(x_1x_0 + x_0x_3) \\
 -& 2\lambda^2_1 g_0^2(x_0x_3 - x_2x_3 - x_1x_2 + x_1x_0) - 2\lambda^2_1 g_2^2(x_2x_3 - x_1x_3 + x_0x_2 - x_0x_1)\\
 -& 2\lambda^2_1 g_1^2(x_2x_1 + x_1x_3 + x_1x_2 - x_0x_3 + x_0x_1)+ \lambda_2\lambda_1 (g_2^2x_0x_3 + g_1^2x_2x_3 \\
 +& g_0^2x_1x_3) + \lambda_2^2 g_0 g_1(x_2x_1 + x_1x_0 + x_0x_2 - \lambda_1) -\lambda_2\lambda^2_1 (3g_0^3g_3 - 2g_0g_1^3   \\ 
-&g_0^2g_2 - 2g_0^3g_1 + g_2- g_1^2 + g_0^2) -\lambda_2(\lambda_1 -\lambda_2) (\lambda_1g_0^2 (g_3^2 +g_1g_3 +g_1^2 +2g_0g_1^3)\\
+ & x_2x_1 +x_1x_0 +x_0x_2).
\end{align*}

Then
\begin{itemize}[leftmargin=*]
\item $H(\lambda_1,\lambda_2, \lambda_3)$ is a lifting of 
$\toba(V) \# \ku G$,

\item any lifting of $\toba(V)\ \# \ku G$ is $\simeq$ to $H(\lambda_1,\lambda_2, \lambda_3)$ for some $(\lambda_1, \lambda_2, \lambda_3) \in \Ss$,

\item $H(\lambda_1,\lambda_2, \lambda_3) \simeq H(\lambda'_1,\lambda'_2, \lambda'_3)$ iff there exists $\mu \in \ku^{\times}$ such that $(\lambda_1,\lambda_2, \lambda_3) = \mu (\lambda'_1,\lambda'_2, \lambda'_3)$.
\end{itemize}

\begin{obs}\label{obs-lifting-a4c2} Let $H$ be the Hopf algebra corresponding  either to  \ref{sec:A4xZ2}.\ref{A4xZ2-tres} or to \ref{sec:A4xZ2}.\ref{A4xZ2-cinco}. The classification of all liftings of $\toba(V) \# H$ follows from \cite[Th. 6.3]{GVay}. Namely, let $J \in \ku G \otimes \ku G$ such that $H\simeq (\ku G)^J$. Then
\begin{itemize}[leftmargin=*]
\item $H(\lambda_1,\lambda_2, \lambda_3)^J$ is a lifting of 
$\toba(V)\ \# H$, for every $(\lambda_1, \lambda_2, \lambda_3) \in \Ss$,

\item any lifting of $\toba(V)\ \# H$ is $\simeq$ to $H(\lambda_1,\lambda_2, \lambda_3)^J$ for some $(\lambda_1, \lambda_2, \lambda_3) \in \Ss$,

\item $H(\lambda_1,\lambda_2, \lambda_3)^J \simeq H(\lambda'_1,\lambda'_2, \lambda'_3)^J$ iff there is $\mu \in \ku^{\times}$ such that $(\lambda_1,\lambda_2, \lambda_3) = \mu (\lambda'_1,\lambda'_2, \lambda'_3)$.
\end{itemize}
\end{obs}

\section{Group-theoretical Hopf algebras over $G = C_7\rtimes_3 C_6$} \label{sec:Afines7}

The classification of the \fd{} Nichols algebras over $G$ is not known, but there are 
$V_3, V_5 \in \ydg$ with $\dim \toba(V_j) = 326592$, $ j = 3,5$, see \cite{G-cm}.  The underlying racks are $\Q_{7, j}$, $j=3, 5$; in both cases, the cocycle is $-1$. Note that $C_7\rtimes_3 C_6\simeq C_7\rtimes_5 C_6$.

There are  two non-trivial semisimple  Hopf algebras of dimension 42, $\mathcal{A}_7(2,3)$ and $\mathcal{A}_7(3,2) \simeq \mathcal{A}_7(2,3)^*$; $G(\mathcal{A}_7(3,2)) \simeq G(\mathcal{A}_7(2,3)) \simeq C_6$ and as coalgebras, $\mathcal{A}_7(2,3) \simeq \ku C_6 \oplus (M_3(\ku)^*)^4$, while $\mathcal{A}_7(3,2) \simeq \ku C_6 \oplus (M_2(\ku)^*)^9$. See  \cite[Chapter 10]{N-memoir}. 
	 
\begin{prop}\label{prop:Afin(7,5)} The non-trivial group-theoretical Hopf algebras over $G$ are $\mathcal{A}_7(2,3)$ and $\mathcal{A}_7(3,2)$.
\end{prop}

\pf Let $(F,\alpha, \Gamma, \beta)$ be a group-theoretical data for $G$; then $F \cap \Gamma$ is trivial. Up to conjugation, either $(F, \Gamma)$ or $(\Gamma, F)$ are isomorphic to  $(C_2, C_{7}\rtimes C_{3})$,  $(C_3,  D_7)$, or $(C_6, C_7)$.
By Lemma \ref{prop:algebras gt equivalentes} \ref{item:cocommutative},  $\gth{G}{C_6}{C_7}{1}{1}$ is trivial, while $H := \gth{G}{C_2}{C_{7}\rtimes C_{3}}{1}{1}$ and 
$H' := \gth{G}{C_3}{D_7}{1}{1}$ are non-trivial. 

Now $H^*$ fits into $\xymatrix@C-10pt{ \Bbbk^{C_{7}\rtimes C_{3}}\ar@{^{(}->}[r] & H^*\ar@{->>}[r] & \Bbbk C_2}$. Since $\Bbbk^{C_{7}\rtimes C_{3}} \simeq \Bbbk C_3 \oplus (M_3(\Bbbk)^*)^2$ as coalgebras,  $H^* \simeq \mathcal{A}_7(2,3)$; then $H\simeq \mathcal{A}_7(3,2)$. Also, $H'^*$ fits into  $\xymatrix@C-10pt{ \Bbbk^{D_{7}}\ar@{^{(}->}[r] & H'^* \ar@{->>}[r] & \Bbbk C_3}$. As  coalgebras, $\Bbbk^{D_{7}} \simeq \Bbbk C_2 \oplus (M_2(\Bbbk)^*)^3$. Hence  $H'^* \simeq \mathcal{A}_7(3,2)$; therefore $H' \simeq \mathcal{A}_7(2,3)$.  \epf

\begin{theorem}\label{prop:c7c6-pointed}	
Each one of  $\mathcal{A}_7(2,3)$ and $\mathcal{A}_7(3,2)$ has two non-zero Yetter-Drinfeld modules $V$ with $\dim \toba(V) =  326592$. 
Thus we get new  Hopf algebras with the dual Chevalley property of dimension 13716864. \qed
\end{theorem}

\section{Group-theoretical Hopf algebras over  $D_n$}\label{sec:dn}
Let $D_n=\langle r,s: r^n=s^2=1, srs=r^{-1}\rangle$ be the dihedral group of order $2n$. 
The classification of the \fd{} pointed Hopf algebras over $D_n$ is known only when  
$n=4t\geqslant 12$, $t\in \N$ \cite{FG}.
Here we deal with the group-theoretical Hopf algebras over $G= D_n$ for every $n \geqslant 3$. We summarize below some well-known facts about $D_n$:

	\begin{center}
		\begin{tabular}{|p{2cm}|c|c|}
		\hline  & $n$ odd & $n$ even
		\\ \hline $Z(D_n)$ & $1$ & $\langle r^{n/2}\rangle$
		
\\\hline 

\multicolumn{1}{|c|}{ $\Aut D_n$} & \multicolumn{1}{c}{$\{\phi_{k, l}:  (k, n)=1, 0\leqslant l< n\}$;} &
 \multicolumn{1}{c|}{$\phi_{k, l}(r)=r^k$, $\phi_{k, l}(s)=sr^l$}
\\\hline 

\multicolumn{1}{|c|}{{\small Subgroups}} & \multicolumn{1}{c}{$\langle r^d\rangle$,  $d| n$;} &
 \multicolumn{1}{c|}{$\langle r^d, r^ks\rangle$,  $d|n$, $0\leqslant k< d$}
			
\\\hline 

\multicolumn{1}{|p{2cm}|}{{\small Subgroups \newline up to $\Aut D_n$}} & \multicolumn{1}{c}{$\langle r^d\rangle$,  $d| n$;} &
 \multicolumn{1}{c|}{$\langle r^d, s\rangle$,  $d|n$}			
			
\\ \hline {\small Normal\newline subgroups} & $\langle r^d \rangle$,  $d|n$ &
$\langle r^d \rangle$,  $d|n$; $\langle r^2, s\rangle$; $\langle r^2, sr \rangle$
			\\ \hline $[D_n, D_n]$ &  $\langle r\rangle$ & $\langle r^2\rangle$
\\ \hline $\widehat{D_n}$ &  $C_2$ & $C_2\times C_2$

			\\ \hline {\small $H^2(D_n,\mathbb C^\times)$} & $1$ & $C_2$ ${}^{(\#)}$

			\\ \hline
		\end{tabular}
\end{center}

\noindent {\small ${}^{(\#)}$
A representative of the non-trivial class is $f_{\chi}\in Z^2(D_n,\ku^\times)$, $ f_{\chi}(r^is^j,r^ks^l)=\chi(r^k)^{j},\,\, j\in\{0, 1\},$  where $\chi:\langle r\rangle\to \ku^\times$ is a character of order $n$. Note that $\chi(r^{\frac{n}{2}})=-1$.}

\subsection{Group-theoretical Hopf algebras over  $G = D_n$, $n$ odd}\label{sec:dn-odd}

\begin{prop}\label{prop:dn-odd} Every  group-theoretical Hopf algebra over $G$ is trivial.
\end{prop}

\pf  Clearly $H^2(\langle r^{d}\rangle, \ku^\times)=1 = H^2(D_{n/d}, \ku^\times)$. Let $(F,\alpha, \Gamma, \beta)$ be a group-theoretical data for $D_n$; then $F \cap \Gamma=1$. By routine arguments, using that $n$ is odd, and Lemma \ref{prop:algebras gt equivalentes}, we see that $\gth{D_n}{F}{\Gamma}{1}{1}$ is trivial.
\epf

\subsection{Group-theoretical Hopf algebras over  $G =D_n$, $n$ even}\label{sec:dn-even}

\begin{prop}\label{prop:dn-even} The classification of the non-trivial group-theoretical Hopf algebras over $G$ is given by the group-theoretical data in Table \ref{tab:gp-th-data-dn-even}.
\end{prop}

\begin{question}
Is it true that \ref{sec:dn}.\ref{dn-even-tres} $\simeq$ \ref{sec:dn}.\ref{dn-even-uno}
$\simeq$ \ref{sec:dn}.\ref{dn-even-cuatro}?
\end{question}

\newcounter{itmdne}
\renewcommand{\theitmdne}{\alph{itmdne}}
\newcommand{\itemdne}[1]{\refstepcounter{itmdne}\ref{sec:dn}.\theitmdne\label{#1}}

\begin{table}[h]
\begin{center}
	\caption{Group-theoretical data for $D_n$, $n$ even}\label{tab:gp-th-data-dn-even}

	\begin{tabular}{|c|p{1.8cm}|c|p{1.8cm}|c|p{2cm}|p{2.4cm}|}
		\hline     $\#$ & $F$ &  $\alpha$ & $\Gamma$ & $\beta$ & Condition & $G(H)$
		\\ \hline
\itemdne{dn-even-tres}  & $\langle r^d,r^ks\rangle$, \newline $d|n$, \newline  $0\leqslant k<d$ \newline $\{d\neq 2$ or \newline $(d\neq n$ and $d\neq \frac{n}{2})\}$ 
    & $1$  & $\langle r^e, s\rangle$,\newline $e|n$, \newline  $\{e\neq 2$ or \newline $(e\neq n$ and $e\neq \frac{n}{2})\}$   &  $1$   & $(d,e)  =2$, \newline $[d,e] = n$,\newline $r^{k} \notin \langle r^{2}\rangle$ & {\small For $d=2$,\newline  $(C_2\times C_2)\rtimes C_2$, if $\frac{n}{2}$ is even,\newline $C_2\times C_2$,  if $\frac{n}{2}$ is odd.\newline For $d\neq 2$,\newline  $C_2\times C_2$,  if $\frac{n}{d}$ is even, \newline  $C_2$, if $\frac{n}{d}$ is odd.}
\\ \hline
\itemdne{dn-even-tres-dual}  & $\langle r^e, s\rangle$,\newline $e|n$,  \newline  $\{e\neq 2$ or \newline $(e\neq n$ and $e\neq \frac{n}{2})\}$
    & $1$  & $\langle r^d,r^ks\rangle$, \newline $d|n$, \newline  $0\leqslant k<d$ \newline $\{d\neq 2$ or \newline $(d\neq n$ and $d\neq \frac{n}{2})\}$    &  $1$   & $(d,e)  =2$, \newline $[d,e] = n$,\newline $r^{k} \notin \langle r^{2}\rangle$ & the same as before for $e$
\\ \hline
\itemdne{dn-even-uno}  & $\langle r^{n/2},s\rangle$    & $\neq 1$  & $D_n$  &  $1$   & $n\neq 4$ & {\small $C_2\times C_2$, if $\frac{n}{2}$ is even \newline $C_2$, if $\frac{n}{2}$ is odd}
\\ \hline
\itemdne{dn-even-uno-dual} & $D_n$  &  1 & $\langle r^{n/2},s\rangle$   & $\neq 1$    & $n\neq 4$ & $C_2\times C_2$
		\\ \hline
	\itemdne{dn-even-cuatro}  & $\langle r^d,r^ks\rangle$, \newline $d|n$, $d\neq 1$ \newline  $0\leqslant k<d$ \newline $\{d\neq 2$ or \newline $(d\neq n$ and $d\neq \frac{n}{2})\}$   & $\neq 1$  &  $\langle r^e, s\rangle$,\newline $e|n$, $e\neq 1$ \newline $\{e\neq 2$ or \newline $(e\neq n$ and $e\neq \frac{n}{2})\}$  &  1  & $(d,e) =1$, \newline $[d,e]  = \frac{n}{2}$,\newline  $r^{k} \in \langle r^{d}\rangle$ & {\small For $d=2$,\newline  $(C_2\times C_2)\rtimes_{\nu} C_2$, if $\frac{n}{2}$ is even,\newline $C_2\times C_2$,  if $\frac{n}{2}$ is odd.\newline For $d\neq 2$,\newline  $C_2\times C_2$,  if $\frac{n}{d}$ is even, \newline  $C_2$, if $\frac{n}{d}$ is odd.}
		\\ \hline	
			\itemdne{dn-even-cuatro-dual}  &  $\langle r^e, s\rangle$,\newline $e|n$, $e\neq 1$\newline $\{e\neq 2$ or \newline $(e\neq n$ and $e\neq \frac{n}{2})\}$  & 1 & $\langle r^d,r^ks\rangle$, \newline $d|n$, $d\neq 1$\newline  $0\leqslant k<d$ \newline $\{d\neq 2$ or \newline $(d\neq n$ and $d\neq \frac{n}{2})\}$    &  $\neq 1$  & $(d,e) =1$, \newline $[d,e]  = \frac{n}{2}$,\newline  $r^{k} \in \langle r^{d}\rangle$ & the same as \ref{sec:dn}.\ref{dn-even-tres} for $e$
		\\ \hline	
	\end{tabular}
	
\end{center}
\end{table}

\pf Let $(F,\alpha, \Gamma, \beta)$ be a group-theoretical datum for $G$. 
Then $F \cap \Gamma$ is either $1$ or else $M = \langle r^{n/2},s\rangle\simeq C_2\times C_2$, up to equivalence of  group-theoretical data.

\noindent\emph{Case 1. } $F \cap \Gamma = 1$, i.e. $(F, \Gamma)$ is an  exact factorization. If $F= \langle r^d\rangle$ and $\Gamma=\langle r^e, r^ks\rangle$, then  $(d,e)=1$, so $de=n$, and  $\gth{D_n}{F}{\Gamma}{1}{1}$ is cocommutative. 
Thus, we may assume that $F=\langle r^d, r^ks\rangle$ and $\Gamma=\langle r^e, s\rangle$, up to equivalence of  group-theoretical data. We claim that this is an exact factorization iff 
\begin{align}\label{eq:dn-even-exfact-2}
(d,e) & =2, & [d,e] & = n, & & r^{k} \notin \langle r^{2}\rangle.
\end{align}
First,  
$F\cap \Gamma  = \big(\langle r^d\rangle \cup  \langle r^d\rangle r^ks\big) \cap \big(\langle r^e\rangle \cup  \langle r^e\rangle s\big)  = \langle r^{[d,e]}\rangle \cup \big(\langle r^d\rangle r^ks \cap \langle r^e\rangle s \big)$.  
So $F\cap \Gamma=1$ iff $n|[d,e]$ and $r^{k} \notin \langle r^{d}\rangle \langle r^{e}\rangle =\langle r^{(d,e)}\rangle$. If $(F, \Gamma)$
is an exact factorization, then   $2n=   |F||\Gamma|=\frac{2n}{d}\cdot\frac{2n}{e}$, i.e. $2n=de$, and $[d,e]=2n$ or $n$. But $[d,e]=2n$ implies $(d,e) =1$ and $r^{k} \in \langle r \rangle$, a contradiction. Thus \eqref{eq:dn-even-exfact-2} holds.
Conversely, if \eqref{eq:dn-even-exfact-2} holds, then $F\cap \Gamma=1$ and 
\begin{align*}
F\Gamma & = \langle r^d\rangle \langle r^e\rangle \cup \langle r^d\rangle   \langle r^e\rangle s \cup \langle r^d\rangle r^ks  \langle r^e\rangle
\cup \langle r^d\rangle r^ks  \langle r^e\rangle s 
\\
&= \langle r^2\rangle \cup \langle r^2\rangle s \cup \langle r^2\rangle r^ks 
\cup \langle r^2\rangle r^{k} = D_n  \text{ since } k \text{ is odd.}
\end{align*}

Finally, $F\lhd D_n$ iff $d=2$; $F$ is abelian iff $d=n$ or $d=\frac{n}{2}$; the same for $\Gamma$. So, we must suppose that $d\neq 2$ or ($d\neq n$ and $d\neq \frac{n}{2}$) and  $e\neq 2$ or ($e\neq n$ and $e\neq \frac{n}{2}$). So, $\gth{D_n}{F}{\Gamma}{\alpha}{\beta}$ is non-trivial. Now, using Lemma \ref{prop:Res-Dn} $\gth{D_n}{F}{\Gamma}{\alpha}{\beta}\simeq \gth{D_n}{F}{\Gamma}{1}{1} $  this gives \ref{sec:dn}.\ref{dn-even-tres}. 
 
\noindent\emph{Case 2. } $F \cap \Gamma = M$.  If $(F, \Gamma)=(\langle r^{n/2},s\rangle, D_n)$, then $F \lhd D_n$ iff $n=4$. By Lemma \ref{prop:algebras gt equivalentes} \ref{item:cocommutative}, $\gth{D_n}{F}{\Gamma}{\alpha}{1}$, $n\neq 4$, is non-trivial, this gives \ref{sec:dn}.\ref{dn-even-uno}. 

Now, we may assume that $F=\langle r^d, r^ks\rangle$ and $\Gamma=\langle r^e, s\rangle$. We claim that this is a factorization such that $F\cap \Gamma=M$ if and only if 
\begin{align}\label{eq:dn-even-exfact-3}
(d,e) & =1, & [d,e] & = \frac{n}{2}, & & r^{k} \in \langle r^{d}\rangle.
\end{align}

We have that $F\cap \Gamma= \langle r^{n/2},s\rangle$ iff $[d,e]\equiv \frac{n}{2}\mod n$ and $r^k\in \langle r^{(\frac{n}{2}, d)}\rangle$. If $(F, \Gamma)$ is a factorization such that $F\cap \Gamma=M$, then $2n=|\frac{|F|\cdot|\Gamma|}{|F\cap \Gamma|}=\frac{1}{4}\frac{2n}{d}\frac{2n}{e}$, i.e. $de=\frac{n}{2}$, and $[d,e]=\frac{n}{2}$, $(d,e)=1$. Thus \eqref{eq:dn-even-exfact-3} holds. Conversely, if \eqref{eq:dn-even-exfact-3} holds, then $F\cap \Gamma=M$ and 
\begin{align*}
F\Gamma & = \langle r^{(d,e)}\rangle \cup \langle r^{(d,e)}\rangle s \cup \langle r^{(d,e)}\rangle r^ks
\cup \langle r^{(d,e)}\rangle r^{k} 
\\
&= \langle r\rangle \cup \langle r\rangle s \cup \langle r\rangle r^ks 
\cup \langle r\rangle r^{k} = D_n.
\end{align*}

If $1\neq \alpha\in H^2( D_{n/d},\ku^\times)\simeq C_2\simeq H^2(D_{n/e} ,\ku^\times) \ni \beta \neq 1$, then $\alpha|_{F\cap \Gamma}\neq 1$,  $\beta|_{F\cap \Gamma}\neq 1$ by Lemma \ref{prop:Res-Dn} below.  Further, $\gth{D_n}{F}{\Gamma}{\alpha}{\beta}$, $d\neq 2$ or ($d\neq n$ and $d\neq \frac{n}{2}$) and  $e\neq 2$ or ($e\neq n$ and $e\neq \frac{n}{2}$), is non-trivial, this gives \ref{sec:dn}.\ref{dn-even-cuatro}.
\epf

Then the Hopf algebra associated to \ref{sec:dn}.\ref{dn-even-uno} is a twisting of $\ku D_n$.

\begin{lema}\label{prop:Res-Dn}
Let $n$ be even. If $F=\langle r^{d},r^ks\rangle$, $d|n$, $0\leqslant k<d$, $\frac{n}{d}$ even, then the restriction map $\operatorname{Res}: H^2(D_n,\mathbb C^\times)\to H^2(F,\mathbb C^\times)$ is non-trivial.
\qed
\end{lema}

\begin{theorem}\label{prop:dn-even-pointed}	
	Let be $n=4t$, $t\geqslant 3$. The Hopf algebras from Table \ref{tab:gp-th-data-dn-even} admit families of \fd{} Yetter-Drinfeld modules $V$ with $\toba(V) = \Lambda (V)$. 
Hence we get new  \fd{} Hopf algebras with the dual Chevalley property. 
\qed
\end{theorem}

The liftings of $\toba(V) \# \ku G$, where $V$ is as above, are classified in \cite[Theorem B]{FG}. Indeed, the Hopf algebras
\begin{itemize}
\item $\toba(M_I)\# \ku D_n$ with $I=\{(i, k)\}\in \mathcal{I}$, $k\neq n$,

\item $\toba(M_L)\# \ku D_n$ with $L \in \mathcal{L}$,

\item $A_I(\lambda, \gamma)$ with $I\in \mathcal{I}$, $|I|\geqslant 1$ or $I=\{(i,n)\}$ and $\gamma\equiv 0$,

\item $B_{I, L}(\lambda, \gamma, \theta, \mu)$ with $(I, L)\in \mathcal{K}$, $|I|>0$ and $|L|>0$,

\end{itemize}
are liftings of $\toba(V) \# \ku D_n$ and any lifting is isomorphic to one of these algebras. For the definitions see \cite[Def. 2.6, 2.9, 2.14, 3.9, 3.11]{FG}. For the isomorphism classes of these families of Hopf algebras see \cite[Lemma 3.16, 3.17]{FG}

\begin{obs} Let be $n=4t$, $t\geqslant 3$ and $H$ be the Hopf algebra corresponding  to  \ref{sec:dn}.\ref{dn-even-uno}. The classification of all liftings of $\toba(V) \# H$ follows from \cite[Th. B]{FG}. The idea is the same as in Remarks \ref{obs-lifting-c3c6}, \ref{obs-lifting-c5c20}, \ref{obs-lifting-s4} and \ref{obs-lifting-a4c2}.
\end{obs}

\end{document}